\RequirePackage{etoolbox}
\csdef{input@path}{{style/}{graphics/}}
\documentclass[aap,MSNbibl,dvips]{arximspdf}

\doi{10.1214/14-AAP1022} 
\volume{25}
\issue{2}
\pubyear{2015}
\firstpage{1030}
\lastpage{1077}

\makeatletter
\newcommand{\rrvert}{\vert}
\newcommand{\llvert}{\vert}
\def\mathbbm{\mathbb}
\newcommand{\eqref}[1]{(\ref{#1})}
\newcommand{\iint}{\int\!\!\!\int}
\newcommand{\defeq}{\mathrel{\mathop:}=}
\newcommand{\eqdef}{=\mathrel{\mathop:}}
\newcommand{\F}{\mathcal{F}}

\newcommand{\uarg}{ \cdot}
\newcommand{\ud}{\mathrm{d}}
\newcommand{\R}{\mathbb{R}}
\newcommand{\N}{\mathbb{N}}
\renewcommand{\P}{\mathbb{P}}
\newcommand{\E}{\mathbb{E}}
\newcommand{\B}{\mathcal{B}}
\newcommand{\esssup}{\mathop{\mathrm{ess\,sup}}}
\newcommand{\given}{:}
\newcommand{\Gap}{\mathop{\operatorname{Gap}}}
\newcommand{\var}{\operatorname{var}}

\newtheorem{theorem}{Theorem}
\newtheorem{corollary}[theorem]{Corollary}
\newtheorem{proposition}[theorem]{Proposition}
\newtheorem{lemma}[theorem]{Lemma}

\newproclaim{definition}[theorem]{Definition}
\newproclaim{assumption}[theorem]{Assumption}
\newproclaim{condition}[theorem]{Condition}

\newproclaim{remark}[theorem]{Remark}
\newproclaim{example}[theorem]{Example}
\makeatother

\begin{document}
\begin{frontmatter}

\title{Convergence properties of pseudo-marginal Markov chain Monte
Carlo algorithms}
\runtitle{Convergence properties of pseudo-marginal MCMC}

\begin{aug}
\author[A]{\fnms{Christophe}~\snm{Andrieu}\thanksref{T1}\ead[label=e1]{C.Andrieu@bristol.ac.uk}}
\and
\author[B]{\fnms{Matti}~\snm{Vihola}\corref{}\thanksref{T2}\ead[label=e2]{matti.vihola@iki.fi}}
\thankstext{T1}{Supported in part by an EPSRC advanced research
fellowship and a Winton Capital research award.}
\thankstext{T2}{Supported by the Academy of Finland (project 250575) and by the
   Finnish Academy of Science and Letters, Vilho, Yrj{\"o} and Kalle
   V{\"a}is{\"a}l{\"a} Foundation.}
\runauthor{C.~Andrieu and M.~Vihola}
\affiliation{University of Bristol and University of Jyv{\"a}skyl{\"a}}
\address[A]{School of Mathematics\\
University of Bristol\\
University Walk\\
Bristol, BS8 1TW\\
United Kingdom\\
\printead{e1}}
\address[B]{Department of Statistics\\
University of Oxford\\
1 South Parks Road\\
Oxford\\
OX1 3TG\\
United Kingdom\\
\printead{e2}}
\end{aug}

\received{\smonth{6} \syear{2012}}
\revised{\smonth{3} \syear{2014}}

%
\begin{abstract} 
We study convergence properties of pseudo-marginal Markov\break chain Monte
Carlo algorithms (Andrieu and Roberts [\textit{Ann.~Statist.}
\textbf{37} (2009) 697--725]). We find that the
asymptotic variance of the pseudo-marginal algorithm is always at
least as large as that of the marginal algorithm. We show that if the
marginal chain admits a (right) spectral gap and the weights (normalised
estimates of the target density) are uniformly bounded, then the
pseudo-marginal chain has a spectral gap. In many cases, a similar
result holds for the absolute spectral gap, which is equivalent to
geometric ergodicity. We consider also unbounded weight distributions
and recover polynomial convergence rates in more specific cases, when
the marginal algorithm is uniformly ergodic or an independent
Metropolis--Hastings or a random-walk Metropolis targeting a
super-exponential density with regular contours. Our results on
geometric and polynomial convergence rates imply central limit
theorems. We also prove that under general conditions, the asymptotic
variance of the pseudo-marginal algorithm converges to the asymptotic
variance of the marginal algorithm if the accuracy of the estimators
is increased.
\end{abstract}
%

%

%
\end{frontmatter}

\section{Introduction} 

Assume that one is interested in sampling from a probability
distribution $\pi$ defined on some measurable space
$(\mathsf{X},\B(\mathsf{X}))$. One practical recipe to achieve
this in complex scenarios consists of using Markov chain Monte Carlo
(MCMC) methods, of which the Metropolis--Hastings update is the main
workhorse \cite{metropolis,hastings}.
We may write the Markov kernel related to a
Metropolis--Hastings algorithm in the form
%
\begin{equation}
P(x,\ud y)\defeq\min \bigl\{ 1,r(x, y) \bigr\} q(x, \ud y)+\delta_{x}(
\ud y)\rho(x), \label{eq:marginal-kernel}
\end{equation}
where $r(x,y)$ is the Radon--Nikodym derivative as defined in
\cite{tierney-note}
%
\begin{equation}
r(x, y) \defeq \frac{\pi(\ud y)q(y, \ud x) }{\pi(\ud x) q(x, \ud y)}
\quad \mbox{and}\quad \rho(x) \defeq1- \int\min\bigl
\{1,r(x,y)\bigr\} q(x,\ud y),\hspace*{-10pt} \label{eq:r-and-rho}
\end{equation}
where $q$ is the so-called proposal kernel (or proposal distribution).
We follow the terminology of \cite{andrieu-roberts} and call this method
the \emph{marginal algorithm}.

In some situations, the marginal algorithm cannot be implemented due
to the intractability of the distribution $\pi$. For example, assuming
that $\pi$ and $q$ have densities (also denoted $\pi$ and $q$) with
respect to some $\sigma$-finite measure, it may be that $\pi$ cannot be
evaluated point-wise, and although $r(x,y)$ may be well defined
theoretically, it cannot be evaluated either. However, in some situations
unbiased nonnegative estimates $\hat{\pi}(x) = W_x\pi(x)$ may be
available; that is, $W_x\sim Q_x(\uarg)\ge0$ and $\E[W_x]=1$ for any
$x\in\mathsf{X}$ (we will refer to $W_x$ as a ``weight'' throughout the
paper). A naive idea may be to use such estimates in place of the true
values in order to compute the acceptance probability. A remarkable
property is that such an algorithm is in fact
correct \cite{andrieu-roberts}.
This can be seen by considering the following probability
distribution:
%
\begin{equation}
\tilde{\pi}(\ud x, \ud w)\defeq \pi(\ud x)\pi_{x}(\ud w) \qquad\mbox{with }
\pi_x(\ud w) \defeq Q_{x}(\ud w)w \label{eq:pseudo-target}
\end{equation}
on the product space $(\mathsf{X}\times\mathsf{W},
\B(\mathsf{X})\times\B(\mathsf{W}))$ where
$\mathsf{W}$ is a Borel subset of $\mathbb{R}_{+}$ and
$\B(\mathsf{W})$ are the Borel
sets on $\mathsf{W}$.
Here $\pi_x(\ud w)$
is a probability measure for each
$x\in\mathsf{X}$, and therefore $\pi$ is a marginal distribution of
$\tilde{\pi}$.

It is possible to implement a Metropolis--Hastings algorithm targeting
$\tilde{\pi}(\ud x, \ud w)$ using a proposal kernel
$\tilde{q}(x,w; \ud y, \ud u) \defeq q(x, \ud y)Q_{y}(\ud u)$ by
defining
%
\begin{eqnarray}\label{eq:def-pseudo-kernel}
&&\tilde{P}(x,w; \ud y, \ud u)
\nonumber
\\[-8pt]
\\[-8pt]
\nonumber
&&\qquad \defeq\min \biggl\{1,r(x, y)\frac{u}{w}
\biggr\} q(x, \ud y)Q_{y}(\ud u ) +\delta_{x,w}(\ud y, \ud
u)\tilde{\rho}(x,w),
\end{eqnarray}
where the probability of rejection is given as
\[
\tilde{\rho}(x,w) \defeq 1 - \iint\min \biggl\{1,r(x, y)\frac{u}{w} \biggr
\} q(x, \ud y)Q_{y}(\ud u ).
\]
This is the \emph{pseudo-marginal algorithm} \cite{andrieu-roberts},
which targets $\pi$ marginally since it is a marginal distribution of
$\tilde{\pi}$, and may be implemented in situations where the marginal
algorithm may not. As a particular instance of the Metropolis--Hastings
algorithm, the pseudo-marginal
algorithm converges to $\tilde{\pi}$ under mild assumptions
(e.g., \cite{nummelin-mcmcs}), and although it may be seen as a ``noisy'' version
of the marginal algorithm, it is exact since it allows us to target
the distribution of interest $\pi$.
The aim of this paper is to study some of the theoretical properties
of such algorithms in terms of the properties of the
weights and those of the marginal algorithm. More precisely we
investigate the rate of convergence of the pseudo-marginal algorithm to
equilibrium and characterise
the approximation of the marginal algorithm by the pseudo-marginal
algorithm in terms of the variability
of their respective ergodic averages.

The apparently abstract structure of the pseudo-marginal algorithm is
in fact shared
by several practical algorithms which have recently been proposed in
order to sample from
intractable distributions. The distribution of $w$ is most often
implicit, as we illustrate now with one of the simplest examples.
Assume for simplicity that the space $\mathsf{X}$ is (a Borel subset of)
$\R^d$, and $\B(\mathsf{X})$ consists of the Borel subsets of
$\mathsf{X}$ and that both $\pi$ and $q(x,\cdot)$ (for any $x\in
\mathsf
{X}$) have densities with respect to the Lebesgue
measure. Consider a situation where the target density is of the form
$\pi(x) = \int\pi(x,z) \,\ud z$ where the integral cannot be
computed analytically. One can suggest approximating this density with
an importance sampling estimate of the integral,
%
\begin{equation}
W_x\pi(x) = \hat{\pi}(x) = \frac{1}{N}\sum
_{n=1}^N \frac{\pi(x, Z_k)}{h_x(Z_k)},\qquad Z_k\sim
h_x(\uarg)\mbox{ independently}, \label{eq:is-w}
\end{equation}
where $h_x$ is a probability density for each $x\in\mathsf{X}$. Note
that it is in fact possible to consider unbiased estimators up to a
normalising constant since such a constant cancels in the acceptance
ratio of the pseudo-marginal algorithm, and without loss of generality,
we will assume this constant to be equal to one throughout.
This setting was considered by Beaumont in the seminal paper
\cite{beaumont} and various extensions proposed in \cite{andrieu-roberts}.
There are more involved applications of this idea. In the context of
state-space models, it has been shown in
\cite{andrieu-doucet-holenstein} that $W_x$ can be obtained with a
particle filter---resulting in ``particle MCMC'' algorithms. In
\cite{beskos-papaspiliopoulos-roberts-fearnhead} it was shown how
exact sampling methods can be used to carry out inference in
discretely observed diffusion models for which the transition
probability is intractable. See also the discussion
\cite{lee-andrieu-doucet} on the connection with pseudo-marginal MCMC
and approximate Bayesian computation.

We now summarise our main findings, which are of two different types,
although some of their
underpinnings and consequences are related.
\subsection*{Rates of convergence} In previous work \cite
{andrieu-roberts} it has been shown
that a pseudo-marginal chain is uniformly ergodic whenever
the marginal algorithm targeting $\pi(x)$ is uniformly ergodic, and the
weights are bounded uniformly in $x$. It was also shown that geometric
ergodicity
is not possible as soon as the weights $W_x$ are unbounded on a set
of positive $\pi$-probability.
We extend the analysis of the convergence
rates of pseudo-marginal algorithms in several directions.

In Section~\ref{13131313131131311}, we show that if the marginal chain admits a
nonzero (right) spectral gap, and the weights are bounded uniformly in $x$,
then the pseudo-marginal chain has also a nonzero spectral gap. Our
proof relies on an explicit lower bound on the spectral gap
(Propositions \ref{prop:gap-p-vs-pbar} and \ref{prop:pseudo-bar-p}).
Our results imply that geometric ergodicity
of a marginal algorithm is inherited by the pseudo-marginal chain
as soon as the weights are uniformly bounded,
either through a slight modification (Remark~\ref{rem:laziness})
or directly in many cases by observing that the pseudo-marginal Markov operator
is positive (Proposition~\ref{prop:positivity}).

We also restate in a more explicit form a result of Andrieu and Roberts
\cite{andrieu-roberts} which establishes the necessity of the
existence of a function $\bar{w}\dvtx\mathsf{X}\rightarrow[0,\infty)$
such that $Q_x ([0,\bar{w}(x)] )=1$ for the geometric ergodicity
of pseudo-marginal algorithms to hold. Assuming that $Q_x$ has
positive mass in any neighbourhood of $\bar{w}(x)$, we show through
specific examples that $\sup_{x\in\mathsf{X}}\bar{w}(x)<\infty$
may in
some cases be a necessary condition for geometric ergodicity of a
pseudo-marginal algorithm to hold (second part of Remark~\ref{rem:rwm-geometrically-ergodic}) while in other situations the
existence of such a uniform upper bound is not a requirement (Remark~\ref{rem:pm-imh-uniformly-ergodic} and the first part of Remark~\ref{rem:rwm-geometrically-ergodic}). Intuitively, the latter will
correspond to situations where the marginal algorithm possesses some
robustness properties which allow it to counter, up to a limit, the
perturbations brought in by the pseudo-marginal approximation.

In Section~\ref{151515151515} we consider the particular case where the
pseudo-marginal algorithm is an
independent Metropolis--Hastings (IMH) algorithm. The primary interest
of this example is pedagogical,
since the corresponding pseudo-marginal implementation is also an IMH,
which lends itself to a straightforward,
yet very instructive, analysis. For example it allows us to establish
that the existence of
(not necessarily uniformly bounded) moments for the weights
leads to polynomial convergence rates, while the existence of
exponential moments leads to sub-exponential rates.

In the light of this pedagogical example, we pursue our analysis by
considering more general scenarios where the
supports of the weight distributions may be unbounded, that is, such
that on some
set of positive $\pi$-probability $Q_x ([0,\bar{w}] )<1$ for any
$\bar{w}<\infty$, implying that the corresponding pseudo-marginal
algorithms cannot be geometric.

In Section~\ref{sec:uniform-marginal}, we only assume that the
marginal algorithm is uniformly ergodic (together with a mild
additional condition) and that the weight distributions are uniformly
integrable. We establish the existence of a Lyapunov function
satisfying a sub-geometric drift condition toward a small set
(Proposition~\ref{prop:w-drift} and Lemma~\ref{lem:uniform-marginal-small-set}). In particular, if the weight
distributions possess finite power moments, we establish polynomial
ergodicity (Corollary~\ref{cor:poly-drift}).

In Section~\ref{171717171717} we consider the popular random-walk
Metropolis (RWM).
Assuming standard tail conditions on $\pi$ which ensure the geometric
ergodicity of the RWM \cite{jarner-hansen} and the existence of uniformly
bounded moments we show that the corresponding pseudo-marginal algorithm
is polynomially ergodic (Theorem~\ref{thm:rwm-bounded-moments}). We
extend this result to the situation where moments of the weights are
assumed to exist but are not
necessarily uniformly bounded in $x$ (i.e., we allow
them to grow in the tails of $\pi$) in Theorem~\ref{thm:rwm-unbounded-moments}. We note in Remark~\ref{rem:rwm-geometrically-ergodic}
that one of the intermediate
results (Lemma~\ref{rem:rwm-geometrically-ergodic}) in fact implies
the existence of a geometric drift when $Q_x ([0,\bar{w}(x)]
)=1$ for
some appropriate function $\bar{w}$, possibly divergent in the tails of
$\pi$, which is a consequence of the fast vanishing assumptions
on the tails of $\pi$.

\subsection*{Asymptotic variance}

It is natural to compare the asymptotic performance of ergodic
averages obtained from a marginal algorithm and its pseudo-marginal
counterpart. One can in fact ask a more general question of practical
relevance. In practice, it is often possible to choose the weight
distributions $Q_x$ from a family $\{Q_x^N\}_{N\in\N}$ indexed by an
accuracy parameter $N$, as for example in \eqref{eq:is-w}. In such
situations $\pi_x^N(\ud w)=Q_x^N(\ud w)w$ converge weakly to
$\delta_1(\ud w)$ as $N\to\infty$, and one may wonder if the asymptotic
variance of the corresponding ergodic averages converge to that of the
marginal algorithm.

In Section~\ref{12121212121212121} we first show that the pseudo-marginal and
marginal algorithms are ordered both in terms of the mean acceptance
probability (Corollary~\ref{cor:acc-prob-order}) and the asymptotic
variance (Theorem~\ref{thm:as-var-order}). The latter result relies on
a generalisation of the argument due to Peskun
\cite{peskun,tierney-note}, which may be of independent interest.
This supports and generalises the empirical observation on
examples that the pseudo-marginal algorithm cannot be more efficient
than its marginal version.

When the weights are uniformly bounded in $x$, we start Section~\ref
{141414141414} with a simple upper bound on the asymptotic variance of
the pseudo-marginal algorithm (Corollary~\ref{cor:autocorr-geom}) from
which it is straightforward to deduce that it converges to that of the
marginal when the weight upper bound goes to one. We generalise this
result to the situation where the weights are unbounded, but
$\pi_x^N(\ud w)$ converges weakly to $\delta_1(\ud w)$ as
$N\rightarrow\infty$ (Theorem~\ref{th:autocorr-conv-coupling}). We
also show how the sub-geometric ergodicity results proved earlier are
essential to establish the conditions of this theorem in practice
(Proposition~\ref{prop:drift-implies-tailcond}).

We conclude in Section~\ref{sec:conclusion} where we briefly discuss
additional implications of our results such as the existence of
central limit theorems, the possibility to compute quantitative
expressions for the asymptotic variance, and the analysis of
generalisations of pseudo-marginal algorithms.


\section{Ordering of the marginal and pseudo-marginal algorithms}
\label{12121212121212121} 

We first introduce some standard notation related to probability measures
and Markov transition probabilities. For $\Pi$ a Markov kernel
and $\mu$ a probability measure defined on some measurable
space $(\mathsf{E},\B(\mathsf{E}))$ and $f$ a measurable
real-valued function on~$\mathsf{E}$, we let for any $x\in\mathsf{E}$,
$\Pi^0 f(x) \defeq f(x)$,
\[
\mu(f) \defeq\int f(x) \mu(\ud x) \quad\mbox{and}\quad \Pi^n f(x) \defeq\int
\Pi(x,\ud y) \Pi^{n-1}f(y)\qquad\mbox{for $n\ge1$.}
\]
We will also denote the inner product between two real-valued functions
$f$ and $g$ on $\mathsf{E}$ as $ \langle f, g  \rangle
_{\mu} \defeq\int f(x)
g(x) \mu(\ud x)$ and the associated norm
$\|f\|_{\mu} \defeq \langle f, f  \rangle_\mu^{1/2}$.

We start by a simple lemma, which plays a key role in the ordering of
the marginal and the pseudo-marginal algorithms.

%
\begin{lemma}
\label{lem:mean-acc-rate} 
For any $x,y\in\mathsf{X}$, we have
\[
\iint Q_{x}(\ud w)w Q_{y}(\ud u) \min \biggl\{ 1,r(x,y)
\frac{u}{w} \biggr\} \leq\min\bigl\{1,r(x,y)\bigr\}.
\]
\end{lemma}
%

\begin{pf} 
Notice that $t\mapsto\min\{1,t\}$ is a concave function. Therefore,
one can apply Jensen's inequality, with the probability measure
$Q_{x}(\ud w)wQ_{y}(\ud u)$, to get the desired inequality.
\end{pf}
%

In order to facilitate the comparison of $P$ and $\tilde{P}$ we follow
\cite{andrieu-roberts} and
introduce an auxiliary transition probability $\bar{P}$ which is
defined on the same space as the pseudo-marginal kernel $\tilde{P}$
and is reversible with respect to $\tilde{\pi}$, 
%
\begin{equation}
\bar{P}(x,w; \ud y, \ud u ) \defeq q(x, \ud y) \pi_{y}(\ud u) \min
\bigl\{1,r(x,y)\bigr\}+\delta_{x,w}(\ud y, \ud u) \rho(x). \label{eq:bar-p}
\end{equation}
Application of Lemma~\ref{lem:mean-acc-rate} leads to the generic result
below, which in turn implies
an order between the expected acceptance rates (Corollary~\ref
{cor:acc-prob-order})
and the asymptotic variances (Theorem~\ref{thm:as-var-order}) of the
marginal and pseudo-marginal algorithms.

%
\begin{proposition}
\label{prop:order-generic} 
Let $g\dvtx\mathsf{X}^2\to[0,\infty)$ be a symmetric measurable
function, that is, such that
$g(x,y)=g(y,x)$ for all $x,y\in\mathsf{X}$. Define
\begin{eqnarray*}
\Delta_{\bar{P}}(g) &\defeq&\int\tilde{\pi}(\ud x, \ud w) \int q(x,\ud y)
\pi_y(\ud u) \min\bigl\{1,r(x,y)\bigr\} g(x,y),
\\
\Delta_{\tilde{P}}(g) &\defeq&\int\tilde{\pi}(\ud x, \ud w) \int q(x,\ud y)
Q_y(\ud u) \min \biggl\{1,r(x,y)\frac{u}{w} \biggr\} g(x,y).
\end{eqnarray*}
Then we have $\Delta_{\bar{P}}(g) \ge\Delta_{\tilde{P}}(g)$ and
whenever these quantities are finite,
\[
\Delta_{\bar{P}}(g) - \Delta_{\tilde{P}}(g) \le\int\pi(\ud x)
Q_{x}(\ud w)\llvert w-1\rrvert \int q(x,\ud y) \min\bigl\{ 1,r(x,y)
\bigr\}g(x,y).
\]
\end{proposition}
%

\begin{pf} 
Denote $a(x,y,u,w) \defeq\min\{1,r(x,y)\} -
\min \{1,r(x,y)\frac{u}{w} \}$.
Since $\int\pi_y(\ud u) = 1 = \int Q_y(\ud u)$, we may write
for a bounded function $g$
\begin{eqnarray*}
\Delta_{\bar{P}}(g) - \Delta_{\tilde{P}}(g) &= &\int\pi(\ud x) q(x,\ud
y) g(x,y)
\int Q_x(\ud w)w Q_y(\ud u) a(x,y,u,w)
\\
&\ge&0,
\end{eqnarray*}
where the inequality is a consequence of Lemma~\ref{lem:mean-acc-rate}.
The general case follows by a truncation argument.

For the second bound, note that
$\min \{1,r(x,y) \frac{u}{w} \} \ge\min\{1,r(x,y)\}
\min \{1,\frac{u}{w} \}$ and $2\min\{u,w\} = u+w - |u - w|$,
and observe that $\Delta_{\tilde{P}}(g)$ can be lower bounded by
\begin{eqnarray*}
&& \int\pi(\ud x)q(x,\ud y)Q_{x}(\ud w) Q_{y}(
\ud u) \min \bigl\{ 1,r(x,y) \bigr\} \min\{u,w\} g(x,y)
\\
&&\qquad= \Delta_{\bar{P}}(g) \\
&&\qquad\quad{}- \frac{1}{2}\int\pi(\ud x)q(x,\ud
y)Q_{x}(\ud w) Q_{y}(\ud u) \min \bigl\{ 1,r(x,y) \bigr\}
|u-w| g(x,y)
\\
&&\qquad\ge\Delta_{\bar{P}}(g) - \int\pi(\ud x) Q_x(\ud w) |1-w| \int
q(x,\ud y) \min\bigl\{1,r(x,y)\bigr\} g(x,y),
\end{eqnarray*}
where the last inequality follows by the bound $|u-w|\le|1 -u | + |1 - w|$,
the symmetry of $g(x,y)$ and because
\[
\pi(\ud x) q(x,\ud y) \min\bigl\{1,r(x,y)\bigr\} = \pi(\ud y) q(y,\ud x) \min
\bigl\{1,r(y,x)\bigr\}. 
\]
\upqed\end{pf}

%
\begin{remark} 
The upper bound $|u-w|\le|1-w| + |1-u|$ used in Proposition~\ref
{prop:order-generic}
adds at most a factor of two, because $\int Q_x(\ud w) |u-w|
\ge|1-w|$.
\end{remark}
%

%
\begin{corollary}
\label{cor:acc-prob-order} 
Let us denote the expected acceptance rates of the marginal and the
pseudo-marginal algorithms as
\begin{eqnarray*}
\alpha_P &\defeq&\int\pi(\ud x) \int q(x,\ud y) \min\bigl\{1,r(x,y)
\bigr\},
\\
\alpha_{\tilde{P}} &\defeq&\int\tilde{\pi}(\ud x, \ud w) \int q(x,\ud y)
Q_y(\ud u) \min \biggl\{1,r(x,y)\frac{u}{w} \biggr\},
\end{eqnarray*}
respectively. Then we have
\[
0 \le\alpha_P - \alpha_{\tilde{P}} \le\int\llvert w-1\rrvert
\pi(\ud x) \bigl(1-\rho(x) \bigr)Q_{x}(\ud w) \le\int\llvert w-1
\rrvert \pi(\ud x)Q_{x}(\ud w).
\]
\end{corollary}
%

\begin{pf} 
Observe first that
\[
\alpha_{\bar{P}} \defeq\int\tilde{\pi}(\ud x, \ud w) \int q(x,\ud y)
Q_y(\ud u) \min\bigl\{1,r(x,y)\bigr\} =\alpha_P.
\]
Applying then Proposition~\ref{prop:order-generic} with $g \equiv1$ implies
\[
0  \le\alpha_P - \alpha_{\tilde{P}} \le\int\llvert w-1\rrvert
\pi(\ud x) \bigl(1-\rho(x) \bigr)Q_{x}(\ud w).
\]
The last inequality follows because $\rho(x)\in[0,1]$ for all
$x\in\mathsf{X}$.
\end{pf}

%
%
\begin{remark}
Corollary~\ref{cor:acc-prob-order} implies also the following bounds:
\[
\alpha_P - \alpha_{\tilde{P}} \le\cases{ \displaystyle\alpha_P
\biggl( \sup_{x\in\mathsf{X}} \int Q_x(\ud w) |1-w| \biggr),
\vspace*{2pt}
\cr
\displaystyle\alpha_P^{1/p} \biggl( \int\pi(\ud x)
Q_x(\ud w) |1-w|^q \biggr)^{1/q}, }
\]
where $p,q>1$ with $1/p+1/q=1$.
\end{remark}

We now define the notion of asymptotic variance for scaled ergodic
averages of a Markov chain.

%
\begin{definition}
\label{def:asvar} 
Let $\Pi$ be a reversible Markov kernel with invariant distribution
$\mu$ defined on some measurable space $(\mathsf{E},\B(\mathsf{E}))$,
and denote by $(X_k)_{k\ge0}$ the corresponding Markov chain
at stationarity, that is, such that $X_0\sim\mu$. Suppose
$f\dvtx\mathsf{E}\to\R$ satisfies $\mu(f^2)<\infty$. The \emph{asymptotic
variance} of $f$ under $\Pi$ is defined as
%
\begin{equation}
\var(f,\Pi) \defeq\lim_{n\to\infty} \frac{1}{n} \E \Biggl(\sum
_{k=1}^n f(X_k) -\mu(f)
\Biggr)^2\in[0,\infty]. \label{eq:def-asvar}
\end{equation}
Whenever the
\emph{integrated autocorrelation time}
\begin{eqnarray}
\tau(f,\Pi) \defeq1 + 2\sum_{k=1}^\infty
\frac{\E[f(X_0)f(X_k)]-\pi(f)^2}{\var_{\mu}(f)}\nonumber\\
\eqntext{\mbox{with } \var_\mu(f) \defeq\mu \bigl(f-\mu(f)\bigr)^2,}
\end{eqnarray}
exists and is finite,
then $\var(f,\Pi) = \tau(f,\Pi) \var_{\mu}(f)\in[0,\infty)$.
\end{definition}

%
Lemma~\ref{lem:asvar-expressions} in Appendix \ref{sec:spectral} shows
that the limit in \eqref{eq:def-asvar} always exists (but may be
infinite) and proves the relation between $\tau(f,\Pi)$ and
$\var(f,\Pi)$. We now show that a pseudo-marginal algorithm is always
dominated by its associated marginal algorithm in terms of asymptotic
variance. The result can be regarded as an extension
of Peskun's approach \cite{peskun,tierney-note}. We point out in the proof
what makes the result not straightforward.

%
\begin{theorem}
\label{thm:as-var-order} 
Assume $f\dvtx\mathsf{X}\to\R$ satisfies $\pi(f^2)<\infty$.
Denote $\var(f,\tilde{P})=\var(\tilde{f},\tilde{P})$ where
$\tilde{f}(x,\uarg)\equiv f(x)$.
\begin{longlist}[(ii)]
\item[(i)]
Then $\var(f,\tilde{P}) \ge\var(f,{P})$.
\item[(ii)]
More specifically,
\[
\var(f,\tilde{P}) \ge \var(f,P) + \liminf_{\lambda\to1-} \bigl[
\Delta_{\bar{P}}(g_\lambda) - \Delta_{\tilde{P}}(g_\lambda)
\bigr],
\]
where $\Delta_{\bar{P}}(g_\lambda)$ and
$\Delta_{\tilde{P}}(g_\lambda)$ are defined in Proposition~\ref{prop:order-generic} and
$g_\lambda(x,y) \defeq[\phi_\lambda(x)-\phi_\lambda(y)]^2$
with
$\phi_\lambda(x) \defeq\sum_{k=0}^\infty\lambda^k [P^k f(x)
- \pi(f)]$ for $\lambda\in[0,1)$.
\end{longlist}
\end{theorem}
%

\begin{pf} 
Our proof is inspired by the proof of Tierney \cite{tierney-note}, Theorem~4, but we cannot use his argument directly
because Proposition~\ref{prop:order-generic} does not apply
to functions depending also on $u$ and $w$.
Observe first from the definition of
$\bar{P}$
that a Markov chain $(\bar{X}_n,\bar{W}_n)_{n\ge0}$
with the kernel $\bar{P}$ and with $(\bar{X}_0,\bar{W}_0)\sim
\tilde{\pi}$ coincides marginally with the marginal chain; that is,
$(X_n)_{n\ge0}$ following $P$ with $X_0\sim\pi$ and
$(\bar{X}_n)_{n\ge0}$ have the same distribution.
Therefore, $\var(f,\bar{P}) = \var(f,P)$. We denote
\[
\bar{f}(x) \defeq f(x) - \pi(f) \in L_0^2(\mathsf{X},
\pi) \defeq\bigl\{f\dvtx\mathsf{X}\to\R\given\pi(f)=0, \pi\bigl(f^2
\bigr)<\infty\bigr\},
\]
and with
a slight abuse of notation define $\bar{f}(x,w) \defeq\bar{f}(x)$
for all
$(x,w)\in\mathsf{X}\times\mathsf{W}$. Notice that $\bar{f}\in
L_0^2(\mathsf{X}\times\mathsf{W},\tilde{\pi})$.
For $\lambda\in[0,1)$, we define the
auxiliary quantities
\[
\var_\lambda(\bar{f},H) = \bigl\langle\bar{f}, (I-\lambda
H)^{-1} (I+\lambda H)f \bigr\rangle_{\tilde{\pi}},
\]
for any Markov kernel $H$ reversible with respect to $\tilde{\pi}$,
where $I$ stands for the identity operator.
We note that from Lemma~\ref{lem:operator-calculus} in Appendix
\ref{sec:spectral}, the quantity $\var_\lambda(\bar{f},H)$ is well
defined and
that from Lemma~\ref{lem:asvar-expressions}, it is sufficient
to show that
$\var_\lambda(\bar{f},\bar{P})\le\var_\lambda(\bar{f},\tilde{P})$
holds for
all $\lambda\in[0,1)$ in order to establish (i).

Using the notation of Lemma~\ref{lem:operator-calculus} with $P_1 =
\bar{P}$ and $P_2 = \tilde{P}$, we can write
%
\begin{eqnarray}\label{eq:var-bound}
\var_\lambda(\bar{f},\tilde{P})- \var_\lambda(\bar{f},\bar{P})
&= &\bigl\langle\bar{f}, A_\lambda(1) \bar{f} \bigr\rangle
_{\tilde{\pi}} - \bigl\langle\bar{f}, A_\lambda(0) \bar{f} \bigr
\rangle_{\tilde
{\pi}}
\nonumber
\\
&= &\int_0^1 \bigl\langle\bar{f},
A'_\lambda(\beta) \bar{f} \bigr\rangle_{\tilde{\pi}} \,\ud
\beta
\\
&=& \int_0^1 \int_0^\beta
\bigl\langle\bar{f}, A''_\lambda(\gamma)
\bar{f} \bigr\rangle _{\tilde{\pi}} \,\ud\gamma\,\ud\beta + \int
_0^1 \bigl\langle\bar{f}, A'_\lambda(0)
\bar{f} \bigr\rangle_{\tilde{\pi
}} \,\ud \beta. \nonumber
\end{eqnarray}
Note that if $\tilde{P}$ and $\bar{P}$ would satisfy
Peskun's order, then the second line is sufficient to conclude
\cite{tierney-note}. We show now
that both terms on the right-hand side of the last line are
nonnegative.

First observe that by Lemma~\ref{lem:operator-calculus},
\begin{eqnarray*}
\bigl\langle\bar{f}, A'_\lambda(0) \bar{f} \bigr\rangle
_{\tilde{\pi}} &=& 2 \lambda \bigl\langle\bar{f}, \smash{(I-\lambda
\bar{P})^{-1} (\tilde{P}-\bar{P}) (I-\lambda\bar{P})^{-1}
\bar{f} } \bigr\rangle_{\tilde{\pi}}
\\
&= &2 \lambda \bigl\langle\vphantom{\bar{f}}\phi_\lambda, \smash {(
\tilde{P}-\bar{P}) \phi_\lambda} \bigr\rangle_{\tilde{\pi}},
\end{eqnarray*}
due to the reversibility of $\bar{P}$, where
$\phi_\lambda\defeq(I-\lambda\bar{P})^{-1} \bar{f}
=\sum_{k=0}^\infty\lambda^k \bar{P}^k \bar{f}$ is well defined by
Lemma~\ref{lem:operator-calculus}.
We notice that $\bar{P}^k \bar{f}(x,w) = P^k \bar{f}(x)$ implying
$\phi_\lambda(x,w) = \phi_\lambda(x)$, and a straightforward
calculation [cf. \eqref{eq:dirichlet-form}] shows that
\begin{eqnarray*}
&&\bigl\langle\phi_\lambda, (\tilde{P}-\bar{P}) \phi_\lambda
\bigr\rangle_{\tilde{\pi}}
\\
&&\qquad= \int\tilde{\pi}(\ud x,\ud w) \phi_\lambda(x) \phi_\lambda(y)
\bigl( \tilde{P}(x,w; \ud y, \ud u) - \bar{P}(x,w; \ud y, \ud u) \bigr)
\\
&&\qquad= \frac{1}{2}\int \bigl(\phi_\lambda(x)- \phi_\lambda(y)
\bigr)^2 \tilde{\pi}(\ud x,\ud w) \bigl( \bar{P}(x,w; \ud y, \ud u) -
\tilde{P}(x,w; \ud y, \ud u) \bigr)
\\
&&\qquad= \frac{1}{2} \bigl[ \Delta_{\bar{P}}(g_\lambda) -
\Delta_{\tilde{P}}(g_\lambda) \bigr],
\end{eqnarray*}
with $g_\lambda(x,y)
=  (\phi_\lambda(x)-\phi_\lambda(y) )^2$, and
Proposition~\ref{prop:order-generic} yields
$ \langle\bar{f}, A'_\lambda(0) \bar{f}  \rangle
_{\tilde{\pi}}\ge0$.
We therefore turn our attention to
\begin{eqnarray*}
&&\bigl\langle\bar{f}, A''_\lambda(\gamma)
\bar{f} \bigr\rangle _{\tilde
{\pi}}
\\
&&\qquad= 4\lambda^2 \bigl\langle\bar{f}, \smash{ (I-\lambda
H_\gamma)^{-1}(\tilde{P}-\bar{P}) (I-\lambda
H_\gamma)^{-1}(\tilde{P}-\bar{P}) (I-\lambda
H_\gamma)^{-1} \bar{f} } \bigr\rangle_{\tilde{\pi}}
\\
&&\qquad= 4\lambda^2 \bigl\langle\varphi, (I-\lambda H_\gamma
\smash{)^{-1}}\varphi \bigr\rangle_{\tilde{\pi}},
\end{eqnarray*}
where $\varphi\defeq(\tilde{P}-\bar{P})(I-\lambda H_\gamma)^{-1}
\bar{f}$, by the reversibility of $\bar{P}$ and $\tilde{P}$ and the
interpolated kernel $H_\gamma= \bar{P} + \gamma(\tilde{P}-\bar{P})$.
It is possible to check that $\varphi\in
L_0^2(\mathsf{X}\times\mathsf{W},\tilde{\pi})$, so we
may conclude (i) by applying Lemma~\ref
{lem:asvar-expressions}
implying $ \langle\varphi, (I-\lambda H_\gamma)^{-1}\varphi
 \rangle_{\tilde{\pi}}\ge0$.

The specific lower bound (ii) follows
from \eqref{eq:var-bound} because the first term is always nonnegative.
\end{pf}
%


\section{Inheritance of the spectral gaps when the weights are
uniformly bound\-ed}
\label{13131313131131311} 

We consider now an order between the spectral gaps
of the pseudo-marginal kernel $\tilde{P}$ and the auxiliary kernel
$\bar{P}$ defined in \eqref{eq:bar-p}.
In particular, we find that if
$w$ is always bounded from above
by $\bar{w}\in[1,\infty)$, that is, $\mathsf{W}=(0,\bar{w}]$,
and $P$ has a nonzero (right) spectral gap (i.e., $P$ is
variance bounding; see \cite{roberts-rosenthal-geometric}, Theorem~14),
then $\tilde{P}$ has a nonzero spectral gap as well. We will also examine
the asymptotic variance constants using the spectral gap bound, and
conclude the section by a discussion on how our results on the spectral
gap can imply geometric ergodicity of $\tilde{P}$.

Suppose $f\dvtx\mathsf{X}\times\mathsf{W}\to\R$ is integrable with
respect to $\tilde{\pi}$.
We denote in this section the function centred with
respect to $w$ as
\begin{eqnarray}
\bar{f}(x,w)\defeq f(x,w)-f_0(x) \nonumber\\
\eqntext{\mbox{with } \displaystyle f_0(x)
\defeq \pi_{x} \bigl(f(x,\uarg) \bigr) = \int_{0}^{\infty}
f(x,w)\pi _{x}(\ud w).}
\end{eqnarray}
The Dirichlet form related to a Markov kernel $\Pi$ with
invariant distribution $\mu$ and a function $g$ is given as
%
\begin{equation}
\mathcal{E}_{\Pi}(g) \defeq \bigl\langle g, (I-\Pi)g \bigr\rangle
_\mu = \frac{1}{2} \int\mu(\ud x) \Pi(x,\ud y) \bigl[g(x)-g(y)
\bigr]^2, \label{eq:dirichlet-form}
\end{equation}
where $I$ is the identity operator.
The spectral gap is defined through
%
\begin{equation}
\Gap(\Pi) \defeq\inf_{g\given\var_\mu(g)>0} \frac{\mathcal{E}_{\Pi}(g)}{\var_\mu(g)} = \inf
_{g\given\mu(g)=0, \|g\|_{\mu}=1} \mathcal{E}_{\Pi}(g), \label{eq:spectral-gap}
\end{equation}
where $\var_\mu(g)$ is given in Definition~\ref{def:asvar}.

%
\begin{proposition}
\label{prop:gap-p-vs-pbar} 
The spectral gap of $\bar{P}$ defined in \eqref{eq:bar-p} satisfies
\[
\Gap(P)\wedge \Bigl(1-\esssup_{x\in\mathsf{X}}\rho(x) \Bigr) \le\Gap(\bar{P})\le
\Gap(P),
\]
where the essential supremum is with respect to $\pi$.
\end{proposition}
%

\begin{pf} 
Let $f\dvtx\mathsf{X}\times\mathsf{W}\to\R$ with
$\tilde{\pi}(f)=0$ and $\|f\|_{\tilde{\pi}}=1$, and compute
\begin{eqnarray*}
\mathcal{E}_{\bar{P}}(f) - \mathcal{E}_{P}(f_0) &=&
\frac{1}{2}\int\pi(\ud x) \pi_x(\ud w) q(x,\ud y)
\pi_y(\ud u) \min\bigl\{1,r(x,y)\bigr\}
\\
&&{}\times \bigl(\bigl[f(x,w)-f(y,u)\bigr]^2-
\bigl[f_0(x)-f_0(y)\bigr]^2 \bigr)
\\
&=& \int\pi(\ud x) \pi_x(\ud w) q(x,\ud y)\min\bigl\{1,r(x,y)\bigr\}
\bigl[f^2(x,w)-f_0^2(x)\bigr]
\\
&=& \int\pi(\ud x) \pi_x(\ud w) \bigl[f(x,w)-f_0(x)
\bigr]^2 \bigl(1-\rho(x) \bigr).
\end{eqnarray*}
In other words,
%
\begin{equation}
\mathcal{E}_{\bar{P}}(f) = \mathcal{E}_{P}(f_0) +
\int\pi(\ud x) \pi_x(\ud w) \bigl(1 - \rho(x)\bigr)
\bar{f}^2(x,w). \label{eq:dirichlet-bar-p}
\end{equation}
If $\var_\pi(f_0)>0$, then we have by
\eqref{eq:dirichlet-bar-p},
%
\begin{eqnarray}
\label{eq:dirichlet-bar-p-bound} \mathcal{E}_{\bar{P}}(f) & \ge&\Gap(P) \var_\pi(f_0)+
\int\pi(\ud x) \pi_x(\ud w) \bigl(1 - \rho(x)\bigr)
\bar{f}^2(x,w)
\nonumber
\\[-8pt]
\\[-8pt]
\nonumber
&\ge&\Gap(P) \bigl(1-\tilde{\pi}\bigl(\bar{f}^2\bigr) \bigr) +
\Bigl(1-\esssup_{x\in\mathsf{X}} \rho(x) \Bigr) \tilde{\pi}\bigl(
\bar{f}^2\bigr),
\nonumber
\end{eqnarray}
where we have used that
$1=\var_{\tilde{\pi}}(f) = \var_\pi(f_0) +
\tilde{\pi}(\bar{f}^2)$ by the variance decomposition identity.
We notice that \eqref{eq:dirichlet-bar-p-bound} holds also when
$\var_\pi(f_0)=0$. We conclude with
the bound $\mathcal{E}_{\bar{P}}(f)
\ge\Gap(P)\wedge (1-\esssup_{x\in\mathsf{X}} \rho(x) )$ which
holds for all $\|f\|_{\tilde{\pi}}=1$ with $\tilde{\pi}(f)=0$,
implying the first inequality.

For the second inequality, note that if $f(x,w)=f_0(x)$ for all
$(x,w)\in\mathsf{X}\times\mathsf{W}$,
then $\pi(f_0)=0$ and $\pi(f_0^2)=1$. Consequently,
$\mathcal{E}_{\bar{P}}(f) = \mathcal{E}_P(f_0)$. Therefore,
$\Gap(\bar{P})\le\Gap(P)$.
\end{pf}
%

%
\begin{remark} 
In the case where $\pi$ is not concentrated on points, that is,\vspace*{1pt}
$\pi(\{x\})=0$ for all $x\in\mathsf{X}$, the statement of Proposition~\ref{prop:gap-p-vs-pbar} simplifies to $\Gap(\bar{P}) = \Gap(P)$,
because then $1-\esssup_{x\in\mathsf{X}} \rho(x) \ge\Gap(P)$ by
Lem\-ma~\ref{lemma:gap-vs-accprob}(ii) in
Appendix~\ref{13131313131131311-lemmas}.
\end{remark}
%

%
\begin{proposition}
\label{prop:pseudo-bar-p} 
Suppose that there exists a constant $\bar{w}\in[1,\infty)$ such that
%
\begin{equation}
Q_x \bigl([0,\bar{w}] \bigr) = 1 \qquad\mbox{for $\pi$-almost every $x\in
\mathsf{X}$.} \label{eq:weight-as-bounded}
\end{equation}
Then, the Dirichlet form of the pseudo-marginal algorithm satisfies
\[
\mathcal{E}_{\tilde{P}}(f) \ge\bar{w}^{-1} \mathcal{E}_{\bar{P}}(f),
\]
for any function with $\tilde{\pi}(f^2)<\infty$, implying
$\Gap(\tilde{P})\geq\bar{w}^{-1}\Gap(\bar{P})$.
\end{proposition}
%

\begin{pf} 
Because $\min\{1,ab\}\ge
\min\{1,a\}\min\{1,b\}$ for all $a,b\ge0$, we have,
denoting $\Delta^2f(x,w;y,u) \defeq [f(x,w)-f(y,u) ]^2$
\begin{eqnarray*}
2\mathcal{E}_{\tilde{P}}(f)& =&\int\tilde{\pi}(\ud x, \ud w) q(x,\ud y)
Q_{y}(\ud u) \min \biggl\{ 1,r(x,y)\frac{u}{w} \biggr\}
\Delta^2f(x,w;y,u)
\\
& \ge& \int
_{u>0} \tilde{\pi}(\ud x, \ud w) q(x,\ud y)
\pi_y(\ud u) \min\bigl\{1,r(x,y)\bigr\}\\
&&{}\times \min \biggl\{
\frac{1}{u},\frac{1}{w} \biggr\} 
\Delta^2f(x,w;y,u)
\\
& \geq&2\bar{w}^{-1}\mathcal{E}_{\bar{P}}(f). 
\end{eqnarray*}
\upqed\end{pf}
%

%
\begin{corollary}
\label{cor:autocorr-geom} 
Assume $\Gap(P)>0$, and there exists some $\bar{w}\in[1,\infty)$ such
that \eqref{eq:weight-as-bounded} holds.
Let $g\dvtx\mathsf{X}\to\R$ satisfy $\pi(g^2)<\infty$.
Then the asymptotic variances (Definition~\ref{def:asvar}) satisfy
\[
\var(g,P) \le \var(g,\tilde{P}) \le\bar{w}\var(g,P) + (\bar{w}-1)
\var_\pi(g), 
\]
where 
$\var(g,\tilde{P})\defeq\var(\tilde{g},\tilde{P})$ with
$\tilde{g}(x,\uarg) \equiv g(x)$.
\end{corollary}
%

\begin{pf} 
Proposition~\ref{prop:pseudo-bar-p} implies
$ \langle f, (I-\smash{\tilde{P}})f  \rangle_{\tilde{\pi
}} \ge
 \langle f, \bar{w}^{-1}(I-\smash{\bar{P}})f  \rangle
_{\tilde{\pi}}$ for all
functions $\tilde{\pi}(f^2)<\infty$, and Lemma~\ref{lem:operator-order}
in Appendix \ref{13131313131131311-lemmas} implies
\[
 \bigl\langle\tilde{g}, (I-\smash{\tilde{P}})^{-1}
\tilde{g} \bigr\rangle _{\tilde
{\pi}} \le\bar{w} \bigl\langle\tilde{g}, (I-
\smash{\bar {P}})^{-1}\tilde {g} \bigr\rangle_{\tilde{\pi}}.
\]
Now note that $\var_{\tilde{\pi}}(\tilde{g})=\var_\pi(g)$ and
$\var(\tilde{g},\bar{P})=\var(g,P)$ hold because
$\bar{P}$ and $P$ coincide marginally; see the proof of
Theorem~\ref{thm:as-var-order}.
The above, together with Theorem~\ref{thm:as-var-order}, imply
\[
\var_{\pi}(g) + \var(g,P) \le \var_{\tilde{\pi}}(\tilde{g}) + \var(
\tilde{g}, \tilde{P}) \le\bar{w} \bigl(\var_{\tilde{\pi}}(\tilde{g})+\var(\tilde
{g},\bar {P}) \bigr),
\]
and allows us to conclude.
\end{pf}
%

%
\begin{remark} 
From the proof of Proposition~\ref{prop:pseudo-bar-p}, one observes
that in fact
\[
\Gap(\tilde{P}) \ge\Gap(\check{P}) \ge\bar{w}^{-1}\Gap(\bar{P}),
\]
where $\check{P}$ is the Markov kernel
with the proposal $q(x,\ud y)Q_y(\ud u)$ and the acceptance
probability $\min\{1,r(x,y)\}\min\{1,u/w\}$ reversible with respect to
$\tilde{\pi}$.
This implies, repeating the arguments in the proof of Corollary~\ref{cor:autocorr-geom},
that $\var(f,\tilde{P})\le\var(f,\check{P})$
for all $\tilde{\pi}(f^2)<\infty$.

We also note that in our follow-up work
\cite{andrieu-vihola-monotone}, we upper bound the spectral gap of the
pseudo-marginal algorithm by that of the marginal,
$\Gap(\tilde{P})\le\Gap(P)$.
\end{remark}
%

Next we show that the boundedness of the support of the weight
distributions $Q_x$ for essentially all $x\in\mathsf{X}$ is a
necessary condition for the spectral gap of the pseudo-marginal
algorithm. The result is similar to Theorem~8 in
\cite{andrieu-roberts}, but its proof is different and the statement
more explicit.

%
\begin{proposition}
\label{prop:geom-bound-necessity} 
If the pseudo-marginal kernel $\tilde{P}$ has a nonzero
spectral gap, then there exists a function
$\bar{w}\dvtx\mathsf{X}\to[1,\infty)$ such that $Q_x ([0,\bar
{w}(x)] )=1$
for
$\pi$-a.e. $x\in\mathsf{X}$.
\end{proposition}
%

\begin{pf} 
We prove the claim by contradiction.
Assume that there exists a set $A\in\B(\mathsf{X})$
with $\pi(A)>0$ such that $Q_x( ([0,\tilde{w}] )<1$ for all
$x\in A$ and all
$\tilde{w}\in[1,\infty)$.
Fix $\varepsilon>0$ and define a measurable function
$\tilde{w}_\varepsilon(x) \defeq\inf\{w\in\N\given1-\tilde{\rho
}(x,w)\le
\varepsilon\}$, which is finite everywhere, because the term
$\tilde{\rho}(x,w)\to1$ as $w\to\infty$ (monotonically) for all
$x\in\mathsf{X}$. Observe that
$\tilde{\pi}(\tilde{A}_\varepsilon)>0$ where
$\tilde{A}_\varepsilon\defeq
\{(x,w)\in A\times\mathsf{W}\given w\ge
\tilde{w}_\varepsilon(x)\}$.
Because $\tilde{w}_\varepsilon$ increases to infinity as $\varepsilon
\to0$,
we have $\tilde{\pi}(\tilde{A}_\varepsilon)\in(0,1/2)$ for small enough
$\varepsilon>0$.
For such $\varepsilon>0$, we may apply
Lemma~\ref{lemma:gap-vs-accprob}(i)
in Appendix \ref{13131313131131311-lemmas}
with the set $\tilde{A}_\varepsilon$,
to conclude that $\Gap(\tilde{P}) \le
(1-\tilde{\pi}(\tilde{A}_\varepsilon))^{-1}\varepsilon\le
2\varepsilon$.
\end{pf}
%

%
\begin{remark}
\label{rem:uniform-bound-necessity} 
Proposition~\ref{prop:geom-bound-necessity} implies the necessity
of the existence of $\bar{w}\dvtx\mathsf{X}\to[1,\infty)$ for
spectral gap and consequently
geometric ergodicity to hold, but does not require the existence
of a uniform upper bound $\bar{w}$ as in Proposition~\ref{prop:pseudo-bar-p}. Uniformity is indeed not necessary as
illustrated in Remarks \ref{rem:pm-imh-uniformly-ergodic} and
\ref{rem:rwm-geometrically-ergodic} with the independent MH and
random walk MH algorithms, respectively; see also \cite{lee-latuszynski},
Remark~1.
However, the second part of Remark~\ref{rem:rwm-geometrically-ergodic} implies that in some cases the
existence of a uniform upper bound $\bar{w}$ is indeed necessary.
\end{remark}
%

The above results are statements on the (right) spectral gap of
$\tilde{P}$ only, which is equivalent to variance bounding property of
$\tilde{P}$ \cite{roberts-rosenthal-variance-bounding}.
In some applications, geometric ergodicity may be
more desirable than variance boundedness.
We first note that in general, geometricity
can be enforced by a slight
algorithmic modification.

%
\begin{remark}
\label{rem:laziness} 
Suppose that $\tilde{P}$ is variance bounding.\vspace*{1pt} Then,
for any $\varepsilon\in(0,1)$, the lazy version of the pseudo-marginal
algorithm $\tilde{P}_\varepsilon\defeq\varepsilon I + (1-\varepsilon
)\tilde{P}$
is geometrically
ergodic \cite{roberts-rosenthal-variance-bounding}, Theorem~2.
\end{remark}
%

In many cases, however, such a modification is unnecessary, because
the pseudo-marginal algorithm can be shown to exhibit also a
nonzero left spectral gap, defined
using the notation in \eqref{eq:spectral-gap}
\[
{ \Gap_L}(\Pi) \defeq\inf_{g:\mu(g)=0, \|g\|_\mu=1} \bigl( 2 +
\mathcal{E}_{\Pi}(g) \bigr) = 1 + \inf_{g:\mu(g)=0, \|g\|_\mu=1} \langle
g, \Pi g\rangle_{\mu}.
\]
Nonzero left and right spectral gaps, or in other words the existence
of an
absolute spectral gap, is equivalent to geometric ergodicity
of a reversible chain (e.g., \cite{roberts-rosenthal-geometric},
Theorem~2.1).

Of particular interest are positive Markov operators $\Pi$
which satisfy\break $\langle g, \Pi g\rangle_{\mu}\ge0$ for all functions
$g$ with
$\|g\|_\mu<\infty$. For positive $\Pi$, clearly $\Gap_L(\Pi)\ge1$ and
establishing geometric ergodicity
only requires focusing on the right spectral gap.
We record the following easy proposition summarising two situations
where the pseudo-marginal algorithm inherits the positivity of the
marginal algorithm.

\begin{proposition}
\label{prop:positivity} 
The pseudo-marginal Markov operator is positive and therefore admits a
left spectral gap in the following cases:
\begin{longlist}[(a)]
\item[(a)]
if the marginal algorithm is an independent
Metropolis--Hastings (IMH);
\item[(b)]
if the marginal algorithm is a random-walk Metropolis (RWM)
with a proposal distribution, which can be written in the form
%
\begin{equation}
q(x,y)=\int\eta(z,x)\eta(z,y)\,\ud z. \label{eq:convolution-proposal}
\end{equation}
\end{longlist}
\end{proposition}
%

\begin{pf} 
Case (a) holds because
the pseudo-marginal version of an IMH is also an IMH
(see also Section~\ref{151515151515}),
which is positive  (e.g., \cite{gasemyr-imh}).
Case (b) follows by using
an argument of Baxendale \cite{baxendale-bounds}, Lemma~3.1, by writing
for $f\dvtx\mathsf{X}\times\mathsf{W}\to\R$ with
$\|f\|_{\tilde{\pi}}<\infty$,
\begin{eqnarray*}
\langle f,\tilde{P}f \rangle_{\tilde{\pi}} &\ge&\int\tilde{\pi}(\ud x,\ud w)
q(x,y) Q_y(\ud u) \min \biggl\{1,\frac{\pi(y)}{\pi(x)}
\frac{u}{w} \biggr\} f(x,w)f(y,u)
\\
&=&\int\phi^{2}(t,z)\,\mathrm{d}t\,\mathrm{d}z\ge0,
\end{eqnarray*}
where $\phi(t,z)\defeq\int f(x,w)\mathbbm{I} \{t\leq\pi
(x)w \}\eta(z,x)Q_{x} (\mathrm{d}w ) \,\ud x$.
\end{pf}
%

%
\begin{remark}
\label{rem:divisability} 
Condition \eqref{eq:convolution-proposal}
holds, in particular, with $q(x,y) =
\tilde{q}(y-x)$ where $\tilde{q}$ is ``divisible;'' that is,
it is the density of the sum of two independent random
variables sharing the same symmetric density $q_0$. Indeed, in such a scenario
$\tilde{q}(y-x) = \int q_0(u) q_0(y-x-u) \,\ud u
= \int q_0(z-x) q_0(y-z) \,\ud z$, and we may take
$\eta(z,x) = q_0(z-x) = q_0(x-z)$. This covers
the case where $\tilde{q}$ is a (possibly multivariate)
Gaussian or Student.
\end{remark}
%

We conjecture that geometric ergodicity is inherited in general
as soon as the weights are uniformly bounded. More precisely,
we believe that if the marginal algorithm is geometrically ergodic
(admits a nonzero absolute spectral gap)
and the weights are uniformly bounded, then the pseudo-marginal
algorithm is also geometrically ergodic. We have not been able to
prove this in general, but we have not found counter-examples either.

For completeness, we, however, provide the following counter-example
which shows that the left spectral gap of the marginal algorithm
may not be inherited by the pseudo-marginal algorithm without the
uniform upper
bound assumption on the weights.

%
\begin{example} 
Let $\mathsf{X}=\N$, $\pi(x) =
2^{-x-1}$ and $q(x,x+1) = q(x,x-1) = 1/2$ for all $x\in\mathsf{X}$.
Direct calculation yields a geometric drift with function
$V(x) = (3/2)^x$ toward an atom $\{0\}$, which shows
that $P$ is geometrically ergodic.

Let us then consider $\tilde{P}$ with the weight distributions $\{Q_x\}
_{x\in\mathsf{X}}$
defined for $x=10^k+n$ with $k\ge1$ and $n\in[1,10^k]$ by
\[
Q_x(w) \defeq(1-\varepsilon_k)\delta_{a(k,n)}(w)
+ \varepsilon_k \delta_{b(k,n)}(w),
\]
and $Q_x(w) \defeq\delta_1(w)$ otherwise,
where $\varepsilon_k \defeq10^{-k}$ and $a(k,n) \defeq2^{-10^k+n}$,
and the constants $b(k,n)\in(1,\infty)$ are chosen so that $Q_x(w)$
have expectation
one. Define the functions
\[
f_k(x,w) \defeq\cases{ +1, &\quad $\mbox{if $x=10^k + n$ with
$n\in\bigl[1,10^k\bigr]$ odd and $w=a(k,n)$},$ \vspace*{2pt}
\cr
-1, &\quad
$\mbox{if $x=10^k + n$ with $n\in\bigl[1,10^k\bigr]$ even
and $w=a(k,n)$},$ \vspace*{2pt}
\cr
0,&\quad $\mbox{otherwise.}$ }
\]
A straightforward calculation shows that
$\lim_{k \rightarrow\infty}
\langle f_k, \tilde{P} f_k \rangle_{\tilde{\pi}} / \|f_k\|_{\tilde
{\pi
}}^2 =-1$,
which shows that there is no left spectral gap.
See \cite{andrieu-vihola-pseudo-arxiv}, Appendix E, for details.
\end{example}
%


\section{Convergence of the asymptotic variance}
\label{141414141414} 

In standard applications of the pseudo-marginal algorithm, one typically
selects $Q_x$ from a family of possible proposal distributions
$Q_{x}^N$ indexed by some precision parameter $N$ which reflects the
concentration of $W$ on $1$.
In most relevant scenarios we are aware of, $N\in\N$ corresponds to
the number
of samples, particles or iterates of an algorithm used to compute an
unbiased estimator of the density
value, as exemplified in \eqref{eq:is-w}. It should be clear that this
is not a restriction. Hereafter, we denote the pseudo-marginal kernels
and the
invariant measures associated with $Q_x^N$ as $\tilde{P}_N$ and
$\tilde{\pi}_N$, respectively.

It is easy to see that if for all $x\in\mathsf{X}$, $Q_x^N(\ud w)w
\to
\delta_1(\ud w)$ as $N\to\infty$ weakly, then $\tilde{\pi}_N(\ud
x,\ud
w)\to\pi(\ud x)\delta_1(\ud w)$ weakly, suggesting that a
pseudo-marginal algorithm with invariant distribution $\tilde{\pi}_N$
may become similar to the marginal algorithm with invariant
distribution $\pi$ as $N\rightarrow\infty$. As pointed out earlier,
whenever $W_x$ is not bounded uniformly, a pseudo-marginal algorithm
cannot be geometric, although its marginal algorithm may be. In fact
it was shown in \cite{andrieu-roberts}, Remark~1, that even in
situations where the weights are uniformly bounded and the
pseudo-marginal algorithm is uniformly ergodic, increasing $N$ may
not improve the rate of convergence of the algorithm, that is, there is
not convergence in terms of rate of convergence.

In this section we, however, show that in many
situations such a convergence takes place in terms of the
asymptotic variance, or equivalently, the integrated autocorrelation
time; see Definition~\ref{def:asvar}. More precisely, we show here that
under simple conditions $\var(g,\tilde{P}_N) \rightarrow\var(g,P)$ as
$N\rightarrow\infty$. We start with a very
simple result, which is a direct consequence of Corollary~\ref{cor:autocorr-geom}.

%
\begin{proposition} 
Suppose that the marginal kernel $P$ has a nonzero spectral gap
and the weight distributions are bounded uniformly in $x\in\mathsf{X}$
by $\bar{w}^N\in(1,\infty)$, that is, $Q_x^N([0,\bar{w}^N])=1$
for all $x\in\mathsf{X}$ and $N\ge N_0$ for some $N_0\in\mathbb
{N}$, and
$\lim_{N\to\infty} \bar{w}^N = 1$.
Then, $\lim_{N\to\infty} \var(g,\tilde{P}_N) = \var(g,P)$
for any $g\dvtx\mathsf{X}\to\R$ with $\pi(g^2)<\infty$.
\end{proposition}
%

\begin{pf} 
The result is direct consequence of Corollary~\ref{cor:autocorr-geom}.
\end{pf}
%

We now extend this result to situations where the distributions
$\{Q_x^N\}_{N\in\N}$ may have an unbounded support, and therefore
$\{\tilde{P}_N\}_{N\in\N}$ may not be geometrically ergodic.
We formulate our result in terms of the following technical
condition assuming uniform convergence of the integrated
autocorrelation series. We will return to this assumption toward the
end of this section and show that it can be
checked in practice with for example Lyapunov type drift conditions;
see Proposition~\ref{prop:drift-implies-tailcond}.

%
\begin{condition}
\label{cond:implicit-autocorr} 
For $g\dvtx\mathsf{X}\rightarrow\mathbb{R}$, suppose that the integrated
autocorrelation time $\tau(g,P)$
(Definition~\ref{def:asvar}) is well defined and finite.
Denote by $(\tilde{X}_k^N)_{k\ge0}$
the Markov chain with initial distribution $\tilde{\pi}_N$
and kernel $\tilde{P}_N$. Assume that
there exists a constant $N_0<\infty$ such that
\[
\lim_{n\to\infty} \sup_{N\ge N_0}\Biggl | \sum
_{k=n}^\infty \E\bigl[ \bar{g}\bigl(
\tilde{X}_0^N\bigr)\bar{g}\bigl(\tilde{X}_k^N
\bigr)\bigr] \Biggr| 
= 0 \qquad\mbox{where } \bar{g} = g -
\pi(g).
\]
\end{condition}

%
The main result of this section is the following:

\begin{theorem}
\label{th:autocorr-conv-coupling} 
Assume that
$g\dvtx\mathsf{X}\to\R$ satisfies $\pi(|g|^{2+\delta})<\infty$,
and Condition \ref{cond:implicit-autocorr} holds for $g$.
Suppose also that
%
\begin{equation}
\lim_{N\to\infty} \int Q_x^N(\ud w) |1-w| =
0 \qquad\mbox{for all $x\in\mathsf{X}$.} 
\label{eq:L1-convergence}
\end{equation}
Then, $\lim_{N\to\infty} \var(g,\tilde{P}_N) = \var(g,P)$.
\end{theorem}
%

\begin{pf} 
If $\var_\pi(g)=0$, the claim is trivial. If $\var_\pi(g)>0$,
our conditions imply that the autocorrelation times exist and are
finite for both the marginal kernel $P$ and the pseudo-marginal kernels
$\tilde{P}_N$ for $N\ge N_0$; this follows from the finiteness of the
terms in the autocorrelation series ensured by the Cauchy--Schwarz
inequality and Condition \ref{cond:implicit-autocorr}.
Therefore,
without loss of generality, we prove the claim for autocorrelation times
$\tau(g,\tilde{P}_N)\to\tau(g,P)$ for a function $g$ with
$\tilde{\pi}_N(g)=\pi(g)=0$ and $\tilde{\pi}_N(g^2)=\pi(g^2)=1$.

Consider the Markov kernels $\bar{P}_N$ defined as in \eqref{eq:bar-p}
with $Q_x^N$ and $\pi_x^N(\ud w)\defeq Q_x^N(\ud w) w$.
Denote by $(\bar{X}_k^N,\bar{W}_k^N)_{k\ge0}$ the corresponding stationary
Markov chain with $(\bar{X}_0^N,\bar{W}_0^N)\sim\tilde{\pi}_N$. Denote
similarly
$(\tilde{X}_k^N,\tilde{W}_k^N)_{k\ge0}$ the stationary
Markov chain corresponding to the kernel
$\tilde{P}_N$ with $(\tilde{X}_0^N,\tilde{W}_0^N)\sim\tilde{\pi}_N$.
Notice that $\bar{P}_N$ and $\tilde{\pi}_N$ coincide marginally with
$P$ and $\pi$, respectively; that is,
$(\bar{X}_k^N)_{k\ge0}$ has the same distribution as that of the stationary
marginal chain $(X_k)_{k\ge0}$ with kernel $P$ and such that $X_0\sim
\pi$.

Choose $\varepsilon\in(0,1)$ and
let $n_0=n_0(\varepsilon)<\infty$ be
such that for all $N\ge N_0$,
%
\begin{equation}
\Biggl|\sum_{k=n_0}^\infty \E\bigl[ g\bigl(
\tilde{X}_0^N\bigr)g\bigl(\tilde{X}_k^N
\bigr)\bigr] \Biggr| \le\varepsilon\quad\mbox{and}\quad
\Biggl |\sum_{k=n_0}^\infty
\E\bigl[ g(X_0)g(X_k)\bigr] \Biggr| \le\varepsilon,
\label{eq:tail-autocorr-sum}
\end{equation}
where the existence of $n_0$ follows from Condition \ref
{cond:implicit-autocorr}.
We have for $N\ge N_0$,
\[
\bigl|\tau(g,P) - \tau(g,\tilde{P}_{N})\bigr|
\le4\varepsilon + 2\Biggl |\sum
_{k=1}^{n_0-1} \E \bigl[ g\bigl(
\tilde{X}_0^N\bigr)g\bigl(\tilde {X}_k^N
\bigr) \bigr] - \E \bigl[g(\bar{X}_0)g(\bar{X}_k) \bigr] \Biggr|.
\]

In order to control the last term,
we consider a coupling argument.
Denote $q\defeq(2 + \delta)/\delta\in(1,\infty)$.
Lemma~\ref{lem:total-variation-bounds} applied with $\check
{\varepsilon}
= \varepsilon n_0^{-q-1}/2$ implies the existence of
$N_1<\infty$ and a set
$\bar{C}\in\B(\mathsf{X})\times\B(\mathsf{W})$ such that
for all $N\ge N_1$,
\begin{eqnarray*}
\tilde{\pi}_N\bigl(\bar{C}^\complement\bigr)&\le&\varepsilon
n_0^{-q-1}/2,
\\
\bigl\|\tilde{P}_N(x,w;\uarg) - \bar{P}_N(x,w;\uarg)\bigr\| & \le&
\varepsilon n_0^{-q-1}/2\qquad \mbox{for all } (x,w)\in\bar{C}.
\end{eqnarray*}
Lemma~\ref{lem:coupling-details} in Appendix \ref{141414141414-lemmas}
applied to $(\tilde{X}^N_k,\tilde{W}^N_k)_{0\le k\le n_0-1}$ and
$(\bar{X}^N_k,\break \bar{W}^N_k)_{0\le k\le n_0-1}$ with the set
$\bar{C}$ shows that the laws of these processes,
$\tilde{\mu}$ and $\bar{\mu}$, respectively,
satisfy the following total variation inequality
for all $N\ge N_1$,
\[
\| \tilde{\mu} - \bar{\mu} \| \le2n_0 \tilde{\pi}_N
\bigl(\bar{C}^\complement\bigr) + n_0 \sup_{(x,w)\in\bar{C}}
\bigl\| \tilde{P}^N(x,w;\uarg) - \bar{P}^N(x,w; \uarg) \bigr\|
\le2 \varepsilon n_0^{-q}.
\]
Therefore, for all $N\ge N_1$, there exists a probability space
$(\bar{\Omega}_N,\bar{\P}_N,\bar{\F}_N)$ where both
$(\tilde{X}_k^N,\tilde{W}_k^N)_{0\le k\le n_0-1}$
and $(\bar{X}_k^N,\bar{W}_k^N)_{0\le k\le n_0-1}$ are defined,
and the set
\[
\bar{A}_N \defeq\bigl\{ \bigl(\tilde{X}_k^N,
\tilde{W}_k^N\bigr)\equiv\bigl(\bar{X}_k^N,
\bar{W}_k^N\bigr), 0 \le k \le n_0-1 \bigr\}
\]
satisfies $\bar{\P}_N(\bar{A}_N^\complement) = \frac{1}{2}\|
\tilde{\mu
} -
\bar{\mu} \| \le\varepsilon n_0^{-q}$ (e.g., \cite{lindvall}, Theorem~5.2).
Denote $p=1+\delta/2$, and note that $p^{-1} + q^{-1} = 1$. Now
for $N\ge N_1$,
\begin{eqnarray*}
&&\Biggl|\sum_{k=1}^{n_0-1} \E \bigl[ g\bigl(
\tilde{X}_0^N\bigr)g\bigl(\tilde {X}_k^N
\bigr) \bigr] - \E \bigl[g\bigl(\bar{X}_0^N\bigr)g\bigl(
\bar{X}_k^N\bigr) \bigr]\Biggr |
\\
&&\qquad=\Biggl |\bar{\E}_N \Biggl[\sum_{k=1}^{n_0-1}
g\bigl(\tilde{X}_0^N\bigr)g\bigl(\tilde{X}_k^N
\bigr) - g\bigl(\bar{X}_0^N\bigr)g\bigl(
\bar{X}_k^N\bigr) \Biggr] \Biggr|
\\
&&\qquad\le \bigl(\bar{\P}_N\bigl(\bar{A}_N^\complement
\bigr) \bigr)^{1/q} \Biggl\{ \Biggl(\bar{\E}_N \Biggl|\sum
_{k=1}^{n_0-1} g\bigl(\tilde{X}_0^N
\bigr)g\bigl(\tilde{X}_k^N\bigr) - g\bigl(
\bar{X}_0^N\bigr)g\bigl(\bar{X}_k^N
\bigr) \Biggr|^p \Biggr)^{1/p} \Biggr\}
\\
&&\qquad\le \bigl(\bar{\P}_N\bigl(\bar{A}_N^\complement
\bigr) \bigr)^{1/q} (n_0-1)\\
&&\qquad\quad{}\times
\max_{1\le
k\le n_0-1} \bigl[ \bigl(\E\bigl| g\bigl(
\tilde{X}_0^N\bigr)g\bigl(\tilde{X}_k^N
\bigr)\bigr|^p \bigr)^{1/p} + \bigl(\E\bigl|g(X_0)g(X_k)\bigr|^p
\bigr)^{1/p} \bigr]
\\
&&\qquad\le2 \varepsilon^{1/q} \bigl(\pi\bigl(|g|^{2+\delta}\bigr)
\bigr)^{1/(2p)},
\end{eqnarray*}
by the H\"{o}lder, Minkowski and Cauchy--Schwarz inequalities.
\end{pf}

Let $\mu_1$ and $\mu_2$ be two probability distributions on
the space $(\mathsf{E},\B(\mathsf{E}))$. We define
the total variation distance
$
\| \mu_1 - \mu_2 \| \defeq
\sup_{|f|\le1} |\mu_1(f)-\mu_2(f)|
= 2 \sup_{0\le f\le1} |\mu_1(f)-\mu_2(f)|
= 2 \sup_{A\in\B(\mathsf{E})}
|\mu_1(A)-\mu_2(A)|$.

%
\begin{lemma}
\label{lem:total-variation-bounds} 
Assume that \eqref{eq:L1-convergence} is satisfied.
Then, for any $\check{\varepsilon}>0$ there exists a
$N_1<\infty$ and a set $\check{C}\in\B(\mathsf{X})\times\B
(\mathsf{W})$
such that
for all $N\ge N_1$,
\begin{eqnarray*}
\tilde{\pi}_N\bigl(\check{C}^\complement\bigr)&\le&\check{
\varepsilon},
\\
\bigl\|\tilde{P}_N(x,w;\uarg) - \bar{P}_N(x,w;\uarg)\bigr\| & \le&
\check{\varepsilon} \qquad\mbox{for all } (x,w)\in\check{C}.
\end{eqnarray*}
\end{lemma}
%

\begin{pf} 
Choose $\check{\varepsilon}>0$, and let $\bar{w} \defeq1 + \check
{\varepsilon}/8$.
It is not difficult to see that assumption \eqref{eq:L1-convergence} implies
for all $x\in\mathsf{X}$,
%
\begin{equation}
\lim_{N\to\infty} \pi_x^N\bigl(\bigl[
\bar{w}^{-1},\bar{w}\bigr]\bigr)  
= 1.
\label{eq:qw-conv}
\end{equation}
Because $\int Q^N_y(\ud u) |1 -u|\le2$,
the dominated convergence theorem together with~\eqref
{eq:L1-convergence} implies
for all $x\in\mathsf{X}$,
%
\begin{equation}
\lim_{N\to\infty} \int q(x,\ud y) Q^N_y(
\ud u) |1-u| = 0. \label{eq:pointwise-convergence}
\end{equation}
By Egorov's theorem, there exists a set $C\in\B(\mathsf{X})$ such that
$\pi(C^\complement)\le\check{\varepsilon}/2$ and the convergence
in both
\eqref{eq:qw-conv} and \eqref{eq:pointwise-convergence}
is uniform in $x$.

For any $x\in\mathsf{X}$, any $w>0$ and any set $A\in
\B(\mathsf{X})\times\B(\mathsf{W})$,
\begin{eqnarray*}
&&\bigl|\tilde{P}_N(x,w;A) - \bar{P}_N(x,w;A)\bigr|
\\
&&\qquad\le2\int q(x,\ud y) Q^N_y(\ud u)\biggl | \min\bigl
\{1,r(x,y)\bigr\} u - \min \biggl\{1,r(x,y)\frac{u}{w} \biggr\} \biggr|
\\
&&\qquad\le2 \int q(x,\ud y) Q^N_y(\ud u) \biggl[|1-u| + \biggl|
\min\bigl\{1,r(x,y)\bigr\} - \min \biggl\{1,r(x,y)\frac{u}{w} \biggr\} \biggr|
\biggr]
\\
&&\qquad\le2 \int q(x,\ud y) Q^N_y(\ud u) \biggl[|1-u| +\biggl |1-
\frac{u}{w} \biggr| \biggr]
\\
&&\qquad\le2 \biggl|1-\frac{1}{w}\biggr | + 4 \int q(x,\ud y) Q^N_y(
\ud u) |1-u|,
\end{eqnarray*}
where the third inequality follows
because
\[
\bigl|\min\{1,ab\} - \min\{1,a\}\bigr| \le\min\{1,a\}|1-b| \qquad
\mbox{for any $a,b\ge0$}.
\]
%
Therefore,
letting $\check{C} \defeq C\times[\bar{w}^{-1},\bar{w}]$,
we can bound the total variation by
\[
\sup_{(x,w)\in\check{C}}\bigl \|\tilde{P}_N(x,w;\uarg) -
\bar{P}_N(x,w;\uarg)\bigr\| \le\frac{\check{\varepsilon}}{2} + 8 \sup
_{x\in C} \int q(x,\ud y) Q_y^N(\ud u)
|1-u|.
\]
Because $\lim_{N\to\infty} \tilde{\pi}_N(\check{C}^\complement)
= \pi(C^\complement)$,
we may conclude by
choosing $N_1<\infty$ such that
$\sup_{x\in C} \int q(x,\ud y) Q_y^N(\ud u) |1-u|\le
\check{\varepsilon}/16$ and
$\tilde{\pi}_N(\check{C}^\complement) \le\check{\varepsilon}$
for all $N\ge N_1$.
\end{pf}
%

%
\begin{remark} 
With additional assumptions in Condition \ref{cond:implicit-autocorr}
and \eqref{eq:L1-convergence} on the rates of convergence, one could obtain
a rate of convergence in Theorem~\ref{th:autocorr-conv-coupling}, that
is find $\{r(n)\}_{n\in\mathbb{N}}$ such that
\[
\bigl|\var(g,\tilde{P}_N)-\var(g,P)\bigr|\le r(N)\to0 \qquad\mbox{as $N\to\infty$},
\]
by going through the proofs of Theorem~\ref{th:autocorr-conv-coupling}
and Lemma~\ref{lem:total-variation-bounds}.
\end{remark}
%

We now provide sufficient conditions implying the conditions
of Theorem~\ref{th:autocorr-conv-coupling}.
Condition \ref{cond:implicit-autocorr} which essentially require quantitative
bounds on the ergodic behaviour of the
pseudo-marginal Markov chains.
Our results rely on polynomial drift conditions which we establish for
some standard algorithms in Sections~\ref{151515151515} and~\ref{171717171717}. Weaker
drift conditions can be shown to imply Condition \ref{cond:implicit-autocorr}
(e.g., \cite{andrieu-fort-explicit,andrieu-fort-vihola-subgeom}), but
we do
not detail this here in order to keep presentation simple.

%
\begin{condition}
\label{cond:unifdrift} 
There exists a function $V\dvtx\mathsf{X}\times\mathsf{W}\to
[1,\infty)$,
a set $C\in\B(\mathsf{X})\times\B(\mathsf{W})$
with $\sup_{(x,w)\in C} V(x,w)<\infty$, constants
$\alpha\in(0,1]$, $b_V\in[0,\infty)$,
$\varepsilon_V\in(0,\infty)$
and
$N_0<\infty$,
such that for all $N\ge N_0$
\[
\tilde{P}_{N}V(x,w)\le V(x,w)-\varepsilon_V
V^{\alpha}(x,w)+b_{V}\mathbbm{I} \bigl\{ (x,w)\in C \bigr\}
\qquad\mbox{$\forall x\in\mathsf{X},w\in\mathsf{W}$,}
\]
and for any $v\in[1,\infty)$, there exists probability measures
$\{\nu^{N}\}_{N\ge N_0}$ and a constant
$\varepsilon_v\in(0,1]$, such that for all $N\ge N_0$,
\[
\tilde{P}_{N}(x,w; \uarg) \ge\varepsilon_v
\nu^{N}(\uarg)\qquad \mbox{for all $(x,w)\in\mathsf{X}\times\mathsf{W}$ such
that $V(x,w)\le v$.}
\]
\end{condition}
%

%
\begin{proposition}
\label{prop:drift-implies-tailcond} 
Assume Condition \ref{cond:unifdrift} holds for the pseudo-marginal
kernels $\tilde{P}_N$, and that for some $\lambda\in[0,1)$ and
$\kappa\in[0,1)$,
\begin{eqnarray*}
\| g\|_{V^{\alpha_{\kappa,\lambda}}} = \sup_{(x,w)
\in\mathsf{X}\times\mathsf{W}} \frac{|g(x)|}
{V^{\alpha_{\kappa,\lambda}}(x,w)} &<&
\infty,
\\
\sup
_{N\ge N_0} \tilde{\pi}_N \bigl(\bigl(|g|+1\bigr) V^{1-\lambda\alpha}
\bigr) &<&\infty,
\end{eqnarray*}
where $\alpha_{\kappa,\lambda} \defeq
\kappa\alpha(1-\lambda)$.
Then Condition
\ref{cond:implicit-autocorr} holds.
\end{proposition}
%

\begin{pf} 
From the assumptions, there exists a finite constant $R$
such that for all $N\ge N_0$ and any $(x,w),(x',w')\in
\mathsf{X}\times\mathsf{W}$,
\begin{eqnarray*}
&&\sum_{k\ge0} r(k) \bigl| \tilde{P}_N^k
g(x,w) - \tilde{P}_N^k g\bigl(x',w'
\bigr) \bigr|
\\
&&\qquad\le R \|g\|_{V^{\alpha_{\kappa,\lambda}}}
\bigl(V^{1-\lambda\alpha}(x,w) +V^{1-\lambda\alpha}
\bigl(x',w'\bigr)-1\bigr),
\end{eqnarray*}
where $r(k)\defeq
(k+1)^{{\alpha(1-\lambda)(1-\kappa)}/{(1-\alpha)}}\to\infty$ as
$k\to\infty$ \cite{andrieu-fort-vihola-subgeom}, Corollary~8; see also~\cite{andrieu-fort-explicit}, Proposition~3.4.
Note that we may write
\begin{eqnarray*}
\bigl|\E_{(x,w)}\bigl[\bar{g}\bigl(\tilde{X}_k^N\bigr)\bigr]\bigr| &=&
\biggl|\tilde{P}_N^k g(x,w) - \int\tilde{\pi}_N(\ud
y,\ud u) \tilde{P}_N^k g(y,u) \biggr|
\\
&\le&\int\tilde{\pi}_N(\ud y,\ud u) \bigl| \tilde{P}_N^k
g(x,w) - \tilde{P}_N^k g(y,u) \bigr|.
\end{eqnarray*}
Therefore, we have for $n\ge0$,
\begin{eqnarray*}
\Biggl|\sum_{k=n}^\infty\E\bigl[ \bar{g}\bigl(
\tilde{X}_0^N\bigr)\bar{g}\bigl(\tilde{X}_k^N
\bigr)\bigr] \Biggr| &\le&\E \Biggl[ \bigl|\bar{g}\bigl(\tilde{X}_0^N
\bigr)\bigr| \sum_{k=n}^\infty \bigl|\E_{(\tilde{X}_0^N,\tilde{W}_0^N)}
\bigl[\bar {g}\bigl(\tilde {X}_k^N\bigr)\bigr] \bigr| \Biggr]
\\
&\le&\frac{\|g\|_{V^{\alpha_{\kappa,\lambda}}}}{r(n)} \bigl[ \tilde{\pi}_N \bigl(|g|
V^{1-\lambda\alpha} \bigr) + \pi\bigl(|g|\bigr) \tilde{\pi}_N\bigl(V^{1-\lambda
\alpha}
\bigr) \bigr].
\end{eqnarray*}
\upqed\end{pf}
%


\section{Sub-geometric ergodicity with an IMH as marginal algorithm}
\label{151515151515} 

The independent Metropolis--Hastings (IMH) algorithm is a specific case
of the Metropolis--Hastings in \eqref{eq:marginal-kernel} corresponding
to a
proposal $q(x,\ud y) = q(\ud y)$ for all $x\in\mathsf{X}$,
such that $\pi\ll q$. It is straightforward to check that
a pseudo-marginal implementation of this algorithm is also an IMH.
This fact allows for the easy exploration of conditions which ensure
uniform and sub-geometric ergodicity of
the pseudo-marginal IMH, and are illustrative of the general ideas we
develop later in the paper. We note that these results may be
relevant, for example, to the analysis of the Particle IMH-EM algorithm
presented in \cite{andrieu-vihola}.

%
\begin{remark}
\label{rem:pm-imh-uniformly-ergodic} 
It is now well known that the IMH is uniformly (and geometrically)
ergodic if and only if
$\pi(\ud x)/q(\ud x)$ is bounded \cite{mengersen-tweedie}.
In the case of the pseudo-marginal IMH,
this is equivalent to assuming that
the ratio $\tilde{\pi}(\ud x,\ud w)/\tilde{q}(\ud x,\ud w)
= w \pi(\ud x)/q(\ud x)$ is bounded; in other words,
assuming that there exists a constant $c\in(0,\infty)$ such that
$Q_x ([0,\bar{w}(x)] )=1$ for
$\pi$-almost every $x\in\mathsf{X}$, where
$\bar{w}(x) \defeq c q(\ud x)/\pi(\ud x)$.
\end{remark}
%

We then give two conditions which ensure sub-geometric ergodicity
of the pseudo-marginal IMH. Our results rely on
Lemma~\ref{lem:imh-subgeom-generic} in Appendix \ref{141414141414-lemmas},
which is inspired by
\cite{jarner-roberts}, which established polynomial ergodicity
and \cite{douc-moulines-soulier-computable}, which explored more
general sub-geometric
rates for the IMH.


%
\begin{corollary}
\label{cor:imh-subgeom-rates} 
Suppose either of the following holds:
\begin{longlist}[(a)]
\item[(a)]
for some $\gamma>0$,
$ \int\tilde{\pi}(\ud x, \ud w)
\exp [ (w\pi(\ud x)/q(\ud x) )^\gamma ] <\infty$,
\item[(b)]
for some $\beta\ge1$,
$ \int\tilde{\pi}(\ud x, \ud w)  (w \pi(\ud x)/q(\ud x)
)^\beta
<\infty$.
\end{longlist}
Then, there exist constants $M,c,c_V\in(0,\infty)$ such that
for $w \pi(\ud x)/q(\ud x)\ge M$, the following drift inequalities hold:
\begin{eqnarray*}
\tilde{P} V_{(a)}(x,w) &\le& V_{(a)}(x,w) - c \kappa
\bigl(V_{(a)}(x,w) \bigr),
\\
\tilde{P} V_{(b)}(x,w) &\le& V_{(b)}(x,w) - c
V_{(b)}^{1-1/\beta}(x,w),
\end{eqnarray*}
respectively,
where $V_{(a)}(x,w) =
\exp ((w\pi(\ud x)/q(\ud x))^{\gamma} )$,
$\kappa(t) = t(\log t)^{-1/\gamma}$ and
$V_{(b)}(x,w) = (w\pi(\ud x)/q(\ud x))^\beta+ 1$.
\end{corollary}
%

\begin{pf} 
Lemma~\ref{lem:imh-subgeom-generic} applied with
(a) $\phi(t) =
\exp(t^\gamma)$ and (b) $\phi(t)=t^\beta+1$.
\end{pf}

The type of drift in Corollary~\ref{cor:imh-subgeom-rates}(a) implies
faster than polynomial sub-geometric rates of convergence
(cf. \cite{douc-fort-moulines-soulier}),
whereas Corollary~\ref{cor:imh-subgeom-rates}(b)
implies polynomial rates of convergence (cf. \cite{jarner-roberts}).
We notice that the result suggests that the pseudo-marginal algorithm
may have a similar rate of convergence as that of the marginal
algorithm.

\section{Sub-geometric ergodicity with uniformly ergodic marginal
algorithm}
\label{sec:uniform-marginal} 

We consider next the situation where the marginal algorithm is
uniformly ergodic. This often corresponds to scenarios where the state space
$\mathsf{X}\subset\R^d$ is compact. It turns out that when the weight
distributions
$\{Q_x\}_{x\in\mathsf{X}}$ do not have bounded supports but are
uniformly integrable,
then the corresponding pseudo-marginal algorithm satisfies a sub-geometric
drift condition toward a set
$\mathsf{C}\defeq\mathsf{X}\times(0,\bar{w}]$ for some
$\bar{w}\in(1,\infty)$. Provided the marginal algorithm satisfies a
practically mild additional condition in
\eqref{eq:continuous-part-condition}, the set $\mathsf{C}$ is
guaranteed to be small for the pseudo-marginal chain.

We start by assuming uniform integrability in a form given by the de la
Vall\'{e}e--Poussin theorem (e.g., \cite{meyer}, page 19 T22). This allows
us to quantify the strength of the sub-geometric drift in a
convenient way, for example, indicating that moment conditions imply
polynomial drifts and consequently polynomial ergodicity.

%
\begin{condition}
\label{cond:ui} 
There exists a nondecreasing convex function
$\phi\dvtx[0,\infty)\to[1,\infty)$ satisfying
\[
\liminf_{t\to\infty} \frac{\phi(t)}{t} = \infty\quad\mbox{and}\quad
M_W \defeq\sup_{x\in\mathsf{X}} \int \phi(w)
Q_{x}(\ud w) < \infty. 
\]
\end{condition}
%

We record a simple implication of Condition \ref{cond:ui}.

\begin{lemma}
\label{lem:quantified-ui} 
Assume Condition \ref{cond:ui} holds. Then, there exists a function
$a(w)\dvtx[0,\infty)\to[0,\infty)$ depending only on $M_W$ and
$\phi$ such that
\[
\sup_{y\in\mathsf{X}} \int_{u\ge w} u
Q_y(\ud u) \le a(w) \quad\mbox{and}\quad \lim_{w\to\infty} a(w)=0.
\]
\end{lemma}
%

\begin{pf} 
For any function $f\dvtx[0,\infty)\to[0,\infty)$ nondecreasing in
$[w,\infty)$, we have
\[
\int_{u\ge w} u Q_y (\ud u) \le\int u
\frac{f(u)}{f(w)} Q_y (\ud u).
\]
The function $f(w)\defeq\phi(w)/w$ is nondecreasing for $w$
sufficiently large,
therefore
\[
\sup_{y\in\mathsf{X}} \int_{u\ge w} u Q_y
(\ud u) \le M_W \frac{w}{\phi(w)} \eqdef a(w)
\mathop{\longrightarrow}\limits^{w\to\infty} 0. 
\]
\upqed\end{pf}
%

The next result establishes a drift away from large values of $w$ for
the pseudo-marginal chain, given that the marginal algorithm has an
acceptance probability uniformly bounded away from zero. All uniformly
(and geometrically) ergodic Markov chains satisfy this
property \cite{roberts-tweedie}, Proposition~5.1.

%
\begin{proposition}
\label{prop:w-drift} 
Suppose that the one-step expected acceptance probability of the
marginal algorithm is bounded away from zero,
\[
\alpha_0 \defeq\inf_{x\in\mathsf{X}} \int q(x,\ud y) \min
\bigl\{1,r(x,y)\bigr\} > 0,
\]
and Condition \ref{cond:ui} holds.

Then, there exist constants $\delta>0$ and $\bar{w}\in(1,\infty)$
such that
\[
\tilde{P} V(x,w) \le V(w) - \delta\frac{V(w)}{w} \mathbbm{I} \bigl\{w\in[
\bar{w},\infty) \bigr\} + M_W \mathbbm{I} \bigl\{w\in(0,\bar{w})
\bigr\}, 
\]
where $V(x,w) \defeq V(w) \defeq\phi(w)$.
The constants $\delta$ and $\bar{w}$ can be chosen to depend on
$\alpha_0$, $\phi$ and $M_W$ only.
\end{proposition}
%

\begin{pf} 
We can estimate
\begin{eqnarray*}
&&\hspace*{-4pt}\tilde{P} V(x,w) - V(w)
\\
&&\qquad\hspace*{-4pt}= \iint q(x,\ud y) Q_y (\ud u) \min \biggl\{1,r(x,y)
\frac{u}{w} \biggr\} \bigl[ \phi(u) - \phi(w) \bigr]
\\
&&\qquad\hspace*{-4pt}\le M_W - \iint q(x,\ud y) Q_y (\ud u) \min \biggl
\{1,r(x,y)\frac
{u}{w} \biggr\} \mathbbm{I} \{u<w \} \bigl[ \phi(w) -
\phi(u) \bigr]
\\
&&\qquad\hspace*{-4pt}\le M_W - \phi(w)\int q(x,\ud y) \min\bigl\{1,r(x,y)\bigr\} \int
_{u<w/2} Q_y (\ud u) \frac{u}{w} \biggl[ 1-
\frac{\phi(w/2)}{\phi(w)} \biggr],
\end{eqnarray*}
because $\min\{1,ab\}\ge\min\{1,a\}\min\{1,b\}$ for all $a,b\ge0$.
The convexity of $\phi$ implies $2\phi(w/2) \le1 + \phi(w)$, and
therefore $\limsup_{w\to\infty} \phi(w/2)/\phi(w) \le1/2$.
Because $\int_{u<w/2}
Q_y (\ud u) u= 1 - \int_{u\ge w/2} Q_y(\ud u) u$, we may apply
Lemma~\ref{lem:quantified-ui}.
Now, for any $\delta_0\in(0,\alpha_0/2)$, there exists
$\bar{w}_0\in(1,\infty)$ such that
\[
\tilde{P} V(x,w) - V(w) \le M_W - \delta_0
\frac{\phi(w)}{w} \qquad\mbox{for all $w\in[ \bar{w}_0,\infty)$.}
\]
The claim follows by taking $\bar{w}\in[\bar{w}_0,\infty)$
sufficiently large
such that
$\phi(w)/w > M_W/\delta_0$ for all $w\in[ \bar{w},\infty)$.
\end{pf}
%

In practice, Condition \ref{cond:ui} is often verified for
moments, that is, $\phi(w) = w^\beta$.
We record the following corollary to highlight the straightforward
connection of $\beta$ to the polynomial drift rate.

%
\begin{corollary}
\label{cor:poly-drift} 
Suppose the conditions of Proposition~\ref{prop:w-drift} hold with
$\phi(w) = w^\beta+ 1$ for some $\beta>1$. Then, the pseudo-marginal
kernel satisfies the drift condition
\begin{eqnarray*}
\tilde{P} V(x,w) &\le V(w) - \delta V^{{(\beta-1)}/{\beta}}(w) + b_V
\mathbbm{I} \bigl\{w\in(0,\bar{w}) \bigr\},
\end{eqnarray*}
where $V(w) \defeq w^\beta+ 1$ and $b_V \defeq M_W + \delta
V^{{(\beta-1)}/{\beta}}(\bar{w})$.
\end{corollary}
%

\begin{pf} 
Follows from Proposition~\ref{prop:w-drift}
observing that $w\le(1+w^{\beta})^{1/\beta} = V(w)^{1/\beta}$.
\end{pf}
%

Proposition~\ref{prop:w-drift} and Corollary~\ref{cor:poly-drift}
establish a drift toward the set $\mathsf{X}\times(0,\bar{w}]$.
They imply sub-geometric convergence of the Markov chain,
with the following lemma showing that the set $\mathsf{X}\times
(0,\bar
{w}]$ is small.

%
\begin{lemma}
\label{lem:uniform-marginal-small-set} 
Denote the (sub-probability) kernel $P_{\mathrm{acc}}(x,A)\defeq
\int_A q(x,\ud y)\times \min\{1,r(x,y)\}$.
Suppose there exists $\varepsilon>0$, an integer $n\in[1,\infty)$
and a
probability measure
$\nu$ on $ (\mathsf{X},\B(\mathsf{X}) )$ such that
for any $A\in\B(\mathsf{X})$,
%
\begin{equation}
P_{\mathrm{acc}}^n(x,A) \ge\varepsilon\nu(A)\qquad \mbox{for all $x\in
\mathsf{X}$.} \label{eq:continuous-part-condition}
\end{equation}
Then, there exists $\bar{w}_0\in(1,\infty)$,
$\tilde{\varepsilon}>0$ and a probability measure $\tilde{\nu}$ on
$ (\mathsf{X}\times\mathsf{W},\B(\mathsf{X})\times\B(\mathsf
{W}) )$
such that for all $\bar{w}\in[\bar{w}_0,\infty)$,
\[
\tilde{P}^n(x,w;\uarg) \ge \frac{\tilde{\varepsilon}}{\bar{w}}
\tilde{\nu}(\uarg)\qquad
\mbox{for all $(x,w)\in\mathsf{X}\times(0,\bar{w}]$.}
\]
\end{lemma}
%

\begin{pf} 
Choose $\bar{w}_0>1$ sufficiently large so that $\varepsilon_W\defeq
\inf_{y\in\mathsf{X}}
\int Q_y(\ud u) \times \min\{\bar{w}_0,u\}>0$; such $\bar{w}_0$ exists due to
Lemma~\ref{lem:quantified-ui} because
\[
\int Q_y(\ud u) \min\{\bar{w}_0,u\} \ge\int
_{u<\bar{w}_0} Q_y(\ud u) u = 1 - \int
_{u\ge\bar{w}_0} Q_y(\ud u) u.
\]
We may write for $A\times B \in\B(\mathsf{X})\times\B(\mathsf{W})$
and for $w\in(0,\bar{w}]$,
\begin{eqnarray*}
\tilde{P}(x,w; A,B) &\ge&\int_A q(x,\ud y) \int
_B Q_y(\ud u) \min \biggl\{1,r(x,y)
\frac{u}{w} \biggr\}
\\
&\ge&\int_A q(x,\ud y) \min\bigl\{1,r(x,y)\bigr\} \int
_B Q_y(\ud u) \min \biggl\{1,
\frac{u}{\bar{w}} \biggr\}
\\
&\ge&\frac{1}{\bar{w}}\int P_{\mathrm{acc}}(x,\ud y) \hat{P}_W(y,
B),
\end{eqnarray*}
where $\hat{P}_W(y, B) = \int_B Q_y(\ud u) \min\{\bar{w}_0,u\}$.
We deduce recursively that
\begin{eqnarray*}
\tilde{P}^n(x,w; A,B) &\ge&\frac{1}{\bar{w}^n} \Bigl[\inf
_{y\in\mathsf{X}} \hat{P}_W\bigl(y,(0,\bar{w}]
\bigr)\Bigr]^{n-1} \int P_{\mathrm{acc}}^n(x,\ud y)
\hat{P}_W(y,B)
\\
&\ge&\frac{\varepsilon_W^{n-1} \varepsilon}{\bar{w}^n} \int_A \nu(\ud y)
\hat{P}_W(y, B) \eqdef\frac{\varepsilon_W^{n-1} \varepsilon}{\bar{w}^n} \tilde{\nu}_0(A
\times B).
\end{eqnarray*}
We may take $\tilde{\nu}(A\times B) = \tilde{\nu}_0(A\times B) /
\tilde{\nu}_0(\mathsf{X}\times\mathsf{W})$ and $\tilde
{\varepsilon} =
\varepsilon\tilde{\nu}_0(\mathsf{X}\times\mathsf{W})>0$.
\end{pf}
%

%
\begin{remark} 
The condition in \eqref{eq:continuous-part-condition} is more stringent
than assuming $P$ uniformly ergodic. However, it is the most common
way to establish the $n$-step minorisation condition $P^n(x,\uarg)\ge
\varepsilon\nu(\uarg)$ in practice, which holds if and only if $P$ is
uniformly ergodic. In the case of a continuous state-space $\mathsf{X}$
where $q(x,\{y\})=0$ and $\nu(\{y\})=0$ for all $x,y\in\mathsf{X}$ and
$n=1$, the condition in \eqref{eq:continuous-part-condition} is in fact
equivalent to $P(x,\uarg) \ge\varepsilon\nu(\uarg)$.
\end{remark}
%


\section{Polynomial ergodicity with a RWM as marginal algorithm}
\label{171717171717} 


We consider next conditions which allow us to establish a polynomial
drift condition
for the pseudo-marginal algorithm in the case where the
marginal algorithm is a geometrically ergodic random-walk Metropolis
(RWM) targeting a super-exponentially decaying target with regular
contours \cite{jarner-hansen}. The existence of such a drift, together with
additional simple assumptions, imply polynomial rates of ergodicity,
but also
Condition \ref{cond:implicit-autocorr} (essential for the convergence
of the pseudo-marginal
asymptotic variance to that of the marginal algorithm) and a central
limit theorem for example.

Our results rely on moment conditions on the distributions
$Q_x(\ud w)$. In Section~\ref{sec:uniform-moment-bounds} we assume
the moments to be (essentially) uniform in $x$, while in
Section~\ref{sec:nonuniform-moment-bounds} we consider the case where
the behaviour of $Q_x(\ud w)$ can get worse as $|x|\to\infty$.
Note that the conditions in Section~\ref{sec:nonuniform-moment-bounds}
may appear more general,
but that they do not include all the cases covered by those of
Section~\ref{sec:uniform-moment-bounds}. This can be seen, for
example, by comparing
Conditions \ref{a:super-exp-regular} and \ref{cond:strong-super-exp-regular}
and the admissible values of $\eta$ in Theorem~\ref{thm:rwm-bounded-moments} and Corollary~\ref{cor:strongly-super-exp}.

It is possible to extend our results beyond the polynomial case. For
example one may assume the existence of exponential moment conditions;
see Remark~\ref{rem:general-moments-rwm}. For the sake of clarity and
brevity, we have opted to detail here the polynomial case only.

%
\begin{remark}
\label{rem:rwm-geometrically-ergodic} 
While our main focus here is on unbounded weight distributions, we
will see that Lemma~\ref{lem:drift-w-bounded-nonuniform} suggests that
geometric ergodicity is still possible when
$Q_x ((0,\bar{w}(x)] )=1$ for all $x\in\R^d$, where
$\bar{w}\dvtx\R^d\to[1,\infty)$ tends to infinity as $|x|\to
\infty$. This
is, however, a consequence of the strong assumption properties on the
tails of $\pi$ which confer the algorithm with a robustness property
with respect to perturbations. Indeed, consider now the RWM on a
compact subset $\mathsf{X}\subset\R^d$ with $\pi$ bounded away from
zero and infinity on $\mathsf{X}$. It is not difficult to establish
that if there does not exist $\bar{w}<\infty$ such that
$Q_x ([0,\bar{w}] )=1$ for $\pi$-almost every $x\in\mathsf{X}$,
then the chain cannot be geometrically ergodic; see, for example, the
proof of Proposition~\ref{prop:geom-bound-necessity}.
\end{remark}
%

Throughout this section, we denote the regions of almost sure
acceptance and possible rejection for the marginal and
pseudo-marginal and algorithms as
\begin{eqnarray*}
A_x &\defeq& \biggl\{ z\in\mathsf{X}\given\frac{\pi(x+z)}{\pi(x)} \ge1
\biggr\},\qquad R_x \defeq A_x^\complement,
\\
A_{x,w} &\defeq& \biggl\{(z,u)\in\mathsf{X}\times\mathsf{W}\given
\frac{\pi(x+z)}{\pi(x)}\frac{u}{w}\ge1 \biggr\}, \qquad R_{x,w} \defeq
A_{x,w}^\complement,
\end{eqnarray*}
respectively, for all $x\in\mathsf{X}$ and $w\in\mathsf{W}$.


\subsection{Uniform moment bounds}
\label{sec:uniform-moment-bounds} 

Consider the following moment condition on the distributions
$\{Q_x\}_{x\in\mathsf{X}}$ where $\mathsf{X}=\R^d$.

%
\begin{condition}
\label{cond:w-moments} 
Suppose there exist constants $\alpha'>0$ and $\beta'>1$
such that
%
\begin{equation}
M_W \defeq\esssup_{x\in\mathsf{X}} \int \bigl(w^{-\alpha'}\vee
w^{\beta'}\bigr) Q_{x}(\ud w) < \infty, \label{eq:moment-cond}
\end{equation}
where $a\vee b\defeq\max\{a,b\}$ and
the essential supremum is taken with respect to the Lebesgue
measure.
\end{condition}
%

We first establish the following simple lemma, used throughout this
section, which guarantees that the moment
condition above holds also for any intermediate exponents.

%
\begin{lemma} 
Given \eqref{eq:moment-cond}, then for all $\alpha\in[0,\alpha']$ and
$\beta\in[0,\beta']$ and any $\gamma\in[-\alpha',\beta]$
\[
\esssup_{x\in\mathsf{X}} \int \bigl(w^{-\alpha}\vee w^{\beta}\bigr)
Q_{x}(\ud w) \le M_W \quad\mbox{and}\quad \esssup_{x\in\mathsf{X}}
\int w^{\gamma} Q_{x}(\ud w) \le M_W.
\]
\end{lemma}
%

\begin{pf} 
The first inequality follows by observing that $w^{-\alpha}\vee
w^{\beta} \le w^{-\alpha'}\vee w^{\beta'}$ for all $w> 0$.
For the second one, suppose first that $\gamma\in[0,\beta']$.
Then $w^{\gamma} \le w^{-\alpha'}\vee w^{\gamma}$, and the result
follows from the first inequality. The case $\gamma\in[-\alpha',0]$ is
similar.
\end{pf}
%

The following condition for the target density $\pi$
was introduced in \cite{jarner-hansen}.

%
\begin{condition}
\label{a:super-exp-regular} 
The target distribution $\pi$ has a density with respect to the
Lebesgue measure (also denoted $\pi$) which is continuously
differentiable and
supported on $\R^d$. The tails of $\pi$ are super-exponentially
decaying and
have regular contours, that is,
\[
\lim_{|x|\to\infty} \frac{x}{|x|} \cdot\nabla\log\pi(x) = -\infty
\quad\mbox{and}\quad \limsup_{|x|\to\infty} \frac{x}{|x|} \cdot
\frac{\nabla\pi(x)}{|\nabla\pi(x)|} < 0,
\]
respectively, where $|x|$ denotes the Euclidean norm of $x\in\R^d$.
Moreover, the proposal distribution satisfies
$q(x, A) = q(A-x) = \int_A q(y-x) \,\ud y$
with a symmetric density $q$
bounded away from zero in some neighbourhood of the origin.
\end{condition}
%

The following theorem establishes a polynomial drift given the
conditions above.

\begin{theorem}
\label{thm:rwm-bounded-moments} 
Suppose $\tilde{P}$ is a pseudo-marginal kernel with distributions
$Q_x(\ud w)$ satisfying Condition \ref{cond:w-moments} with some
constants $\alpha'>0$ and $\beta'>1$, and that the corresponding
marginal algorithm is a random walk Metropolis with invariant
density $\pi$ and proposal density $q$ satisfying Condition
\ref{a:super-exp-regular}.

Define $V\dvtx\mathsf{X}\times\mathsf{W}\to[1,\infty)$ as
follows:
%
\begin{equation}
V(x,w) \defeq c_{\pi}^{\eta} \pi^{-\eta}(x)
\bigl(w^{-\alpha}\vee w^{\beta}\bigr)\qquad \mbox{where } c_{\pi}
\defeq\sup_{z\in\mathsf{X}} \pi(z), \label{eq:def-drift-function}
\end{equation}
for some constants $\eta\in(0,\alpha'\wedge1)$, $\alpha\in(\eta,\alpha']$
and $\beta\in(0,\beta'-\eta)$.

Then, there exists constants $\bar{w}, M,b\in[1,\infty)$,
$\underline{w}\in(0,1]$ and $\delta_V>0$ such that
%
\begin{equation}
\tilde{P} V(x,w) \le \cases{V(x,w) - \delta_V
V^{{(\beta-1)}/{\beta}}(x,w), &\quad $\mbox{for all } (x,w)\notin\mathsf{C}$, \vspace*{2pt}
\cr
b, &\quad $\mbox{for all $(x,w)\in\mathsf{C}$,}$} \label{eq:def-drift-condition}
\end{equation}
where $\mathsf{C}\defeq\{
(x,w)\in\mathsf{X}\times\mathsf{W}\given|x|\le M,
w\in[\underline{w},\bar{w}] \}$.

Moreover, $b$, $\delta_V$ and $\mathsf{C}$ depend only on the marginal
algorithm, the constants $\alpha',\beta'$ and $M_W$ in Condition
\ref{cond:w-moments} and the chosen constants $\alpha,\beta,\eta$.
\end{theorem}
%

\begin{pf} 
Let $\bar{w}\in[1,\infty)$ and $\delta_V'>0$ be as in Lemma~\ref{lem:drift-w-large}, so that
$\tilde{P} V(x,w) \le V(x,w) - \delta_V' V^{{(\beta-1)}/{\beta
}}(x,w)$ for
all $x\in\mathsf{X}$ and all $w\ge\bar{w}$.
Then apply Lemma~\ref{lem:drift-w-bounded} with the fixed
value of $\bar{w}$ to obtain a $M\in[1,\infty)$ and $\lambda\in[0,1)$
such that
%
\begin{equation}
\tilde{P} V(x,w) \le\lambda V(x,w) = V(x,w) - (1-\lambda)V(x,w), \label{eq:drift-geom-v}
\end{equation}
for all $w\in(0,\bar{w}]$ and $|x|\ge M$.
Lemma~\ref{lem:bounded-in-compacts} implies that \eqref{eq:drift-geom-v}
holds with all $x\in\mathsf{X}$ and $w\in(0,\underline{w}]$,
with some
$\lambda'\in[0,1)$. Because $V\ge1$, we conclude
the claim for $(x,w)\notin\mathsf{C}$ with $\delta_V\defeq
\min\{\delta_V',1-\lambda,1-\lambda'\}$.
Lemma~\ref{lem:bounded-in-compacts} implies the case $(x,w)\in
\mathsf{C}$.

The dependence on $b$, $\delta_V$ and $\mathsf{C}$ is clear from
the proofs of Lemmas
\ref{lem:drift-w-bounded} and \ref{lem:bounded-in-compacts}.
\end{pf}
%

%
\begin{remark}
\label{rem:general-moments-rwm} 
It is possible to generalise Theorem~\ref{thm:rwm-bounded-moments} for
nonpolynomial moments. In particular, we may let $V(x,w) = c_\pi^\eta
\pi^{-\eta}(x) \phi(w)$ where $\phi\dvtx(0,\infty)\to[1,\infty
)$ is
defined by
\[
\phi(w) \defeq\cases{ a(w), &\quad $w\in(0,1],$ \vspace*{2pt}
\cr
b(w), &\quad $w\in(1,
\infty)$,}
\]
with nonincreasing $a\dvtx(0,1]\to[1,\infty)$ and
nondecreasing $b\dvtx(1,\infty)\to[1,\infty)$ satisfying
\[
\lim_{w\to0+} w^{-\eta}a(w) = \infty \quad\mbox{and}\quad \lim
_{w\to\infty} b(w)/w = \infty,
\]
and for some $\gamma>\eta$
\[
\esssup_{x\in\mathsf{X}} \int_0^1 a(w)
Q_x(\ud w)<\infty \quad\mbox{and}\quad \esssup_{x\in\mathsf{X}} \int
_1^\infty b(w) w^\gamma Q_x(
\ud w) < \infty.
\]
Note that $a(w)$ and $b(w)$ must grow at least
polynomially as $w\to0+$ and $w\to\infty$, respectively.
For example, $b(w)=\exp(c_b w)$ allows one to establish the claim with
the stronger drift condition
\[
\tilde{P} V(x,w) \le V(x,w) - \hat{\delta}_V \frac{V(x,w)}{\log
\circ V(x,w) }
\qquad(x,w)\notin\mathsf{C},
\]
instead of the polynomial drift in \eqref{eq:def-drift-function}.
\end{remark}
%

We conjecture that the negative moment condition and the presence of
$w^{-\alpha}$ in
the drift function are not necessary in order to establish polynomial
ergodicity in general. It seems, however,
difficult to establish a one-step drift condition without any control
of the
behaviour of the distributions $Q_x$ near zero.

We first consider a simple result which is auxiliary to the other lemmas.

%
\begin{lemma}
\label{lem:acc-prob} 
We have the following bounds for all $x,z\in\mathsf{X}$, $w>0$,
$\hat{\alpha}>0$, and $\hat{\beta}>1$:
%
\begin{eqnarray*}\mathrm{(i)}&&\quad\int\min \biggl\{1,\frac{u}{w} \biggr\}
Q_x(\ud u) \ge
\frac{1}{w} \biggl(1 - \frac{1}{w^{\hat{\beta}-1}} \int u^{\hat{\beta}}
Q_x(\ud u) \biggr), 
\\
\mathrm{(ii)}&&\quad\int_{\{u\given(z,u)\in A_{x,w}\}} Q_{x+z}(\ud u) \ge1 - w^{\hat{\alpha}}
\biggl(\frac{\pi(x)}{\pi(x+z)} \biggr)^{\hat{\alpha}} \int u^{-\hat{\alpha}}
Q_{x+z}(\ud z). 
\end{eqnarray*}
\end{lemma}
%

\begin{pf} 
The bound (i) follows by writing
\begin{eqnarray*}
\int\min \biggl\{1,\frac{u}{w} \biggr\} Q_x(\ud u) &=&
\frac{1}{w} \biggl(1-\int_{u\ge w}(u-w) Q_x(
\ud u) \biggr)
\\
&\ge&\frac{1}{w} \biggl(1 - \int_{u\ge w}
uQ_x(\ud u) \biggr),
\end{eqnarray*}
and using the estimate $\mathbbm{I} \{u\ge w \}\le
(u/w)^{\hat{\beta}-1}$.
For (ii), similarly
\[
\int_{\{u\given(z,u)\in A_{x,w}\}} Q_{x+z}(\ud u) = 1 - \int
_{ \{u < w ({\pi(x)}/{(\pi(x+z))})  \}} Q_{x+z}(\ud u)
\]
and use $\mathbbm{I} \{u < w \vphantom{\int}\smash{\frac{\pi
(x)}{\pi (x+z)}} \}\le u^{-\hat{\alpha}}
 (w\frac{\pi(x)}{\pi(x+z)} )^{\hat{\alpha}}$.
\end{pf}
%

We next consider the case where $w$ is large,
and establish a polynomial drift in this case.

%
\begin{lemma}
\label{lem:drift-w-large} 
Suppose the conditions of Theorem~\ref{thm:rwm-bounded-moments} hold.
Then, there exist constants $\delta_V>0$ and
$\bar{w}\in[1,\infty)$
such that
\[
\tilde{P} V(x,w) \le V(x,w) - \delta_V V^{(\beta-1)/{\beta}}(x,w)\qquad
\mbox{for all $x\in\mathsf{X}$ and $w\in[\bar{w},\infty)$.}
\]
\end{lemma}
%

\begin{pf} 
We may write for $w\ge\bar{w}\ge1$
\begin{eqnarray*}
\frac{\tilde{P} V(x,w)}{V(x,w)} &=& \iint_{A_{x,w}} a_{x,w}(z,u)
Q_{x+z}(\ud u)q(\ud z)
\\
&&{}+ \iint_{R_{x,w}} b_{x,w}(z,u)
Q_{x+z}(\ud u)q(\ud z),
\end{eqnarray*}
where
%
\begin{eqnarray}
a_{x,w}(z,u)&\defeq& \biggl( \frac{\pi(x)}{\pi(x+z)} \biggr)^{\eta}
\frac{u^{-\alpha} \vee u^{\beta}}{w^{\beta}}, \label{eq:a-w-large}
\\
b_{x,w}(z,u)&\defeq& \biggl(\frac{\pi(x+z)}{\pi(x)} \biggr)^{1-\eta}
\frac{u^{1-\alpha}\vee u^{1+\beta}}{w^{1+\beta}} + \biggl(1-\frac{\pi(x+z)}{\pi(x)} \frac{u}{w} \biggr).
\label{eq:b-w-large}
\end{eqnarray}
We now estimate both integrals by partitioning their integration
domains into their intersections with the
acceptance and the rejection sets of the marginal algorithm. For
notational simplicity we denote
$A_{x,w} \cap R_x=A_{x,w} \cap(R_x \times\mathsf{W})$ etc.

The bound for the first integral is straightforward,
\[
\iint_{A_{x,w}\cap A_x} a_{x,w}(z,u) Q_{x+z}(\ud u)q(\ud z)
\le\frac{M_W}{w^{\beta}}.
\]
For the second one, observe that $1\le
 (\frac{\pi(x+z)}{\pi(x)}\frac{u}{w} )^{\eta}$ on
$A_{x,w}$, implying
\begin{eqnarray*}
&&\iint_{A_{x,w}\cap R_x} a_{x,w}(z,u) Q_{x+z}(\ud u)q(\ud z)
\\
&&\qquad\le\frac{1}{w^{\beta+\eta}} \iint_{A_{x,w}\cap R_x} u^{\eta-\alpha}\vee
u^{\eta+\beta} Q_{x+z}(\ud u) q(\ud z) \le\frac{M_W}{w^{\beta+\eta}},
\end{eqnarray*}
because $\beta+\eta\le\beta'$.
Similarly, because
$ (\frac{\pi(x+z)}{\pi(x)}\frac{u}{w} )^{1-\eta}\le1$ on
$R_{x,w}$ we have
\begin{eqnarray*}
&&\iint_{R_{x,w}} \biggl(\frac{\pi(x+z)}{\pi(x)} \biggr)^{1-\eta}
\frac{u^{1-\alpha}\vee u^{1+\beta}}{w^{1+\beta}} Q_{x+z}(\ud u) q(\ud z)
\\
&&\qquad\le\frac{1}{w^{\beta+\eta}} \iint_{R_{x,w}} u^{\eta-\alpha}\vee
u^{\eta+\beta} Q_{x+z}(\ud u) q(\ud z) \le\frac{M_W}{w^{\beta+\eta}}.
\end{eqnarray*}

We now turn to the crucial remainder, which approaches unity
as $w$ grows.
\begin{eqnarray*}
&&\iint_{R_{x,w}} \biggl(1-\frac{\pi(x+z)}{\pi(x)}\frac{u}{w} \biggr)
Q_{x+z}(\ud u) q(\ud z)
\\
&&\qquad= 1 - \iint\min \biggl\{1,\frac{\pi(x+z)}{\pi(x)}\frac
{u}{w} \biggr\}
Q_{x+z}(\ud u) q(\ud z)
\\
&&\qquad\le1 - \iint \min \biggl\{1,\frac{\pi(x+z)}{\pi(x)} \biggr\}\min \biggl\{1,
\frac
{u}{w} \biggr\} Q_{x+z}(\ud u) q(\ud z)
\\
&&\qquad\le1 - \frac{\nu}{w}\int_{\{z\given({\pi(x+z)}/{\pi(x)})\ge
\nu\}} \biggl(1-
\frac{M_W}{w^{\beta'-1}} \biggr)q(\ud z),
\end{eqnarray*}
by Lemma~\ref{lem:acc-prob}(i), where $\nu\in(0,1)$.
Lemma~\ref{lem:containment-and-delta-acc}(ii)
in Appendix \ref{171717171717-lemmas} implies the existence of a $\nu>0$
such that
$\inf_{x\in\mathsf{X}} q (\{z\given\frac{\pi(x+z)}{\pi(x)}\}
\ge
\nu )>0$.
Therefore, there exists a $\nu_2\in(0,\nu)$,
such that whenever $w$ is sufficiently large
\[
\iint_{R_{x,w}} \biggl(1-\frac{\pi(x+z)}{\pi(x)}\frac{u}{w} \biggr)
Q_{x+z}(\ud u) q(\ud z) \le1 -\frac{\nu_2}{w}.
\]
Because $\beta>1$, the terms of the order $w^{-\beta}$ or
$w^{-\eta-\beta}$ vanish faster than $w^{-1}$ when $w$ increases.
Consequently, we have
for any $\nu_3\in(0,\nu_2)$,
by further assuming $w$ sufficiently large, that
\begin{eqnarray*}
\tilde{P}V(x,w) &\le& \biggl(1- \frac{\nu_3}{w} \biggr)V(x,w)
\\
&=& V(x,w) - \nu_3 V^{\kappa}(x,w) \bigl(c_\pi
\pi^{-\eta}(x) \bigr)^{1-\kappa} \le V(x,w) - \nu_3
V^{\kappa}(x,w),
\end{eqnarray*}
where $\kappa= \frac{\beta-1}{\beta}\in(0,1)$.
\end{pf}
%

Next we deduce that in the regime where $w$ is bounded,
we have a geometric drift.

%
\begin{lemma}
\label{lem:drift-w-bounded} 
Assume the conditions of Theorem~\ref{thm:rwm-bounded-moments} hold,
and let $\bar{w}\in[1,\infty)$.
Then, there exist constants $\lambda\in[0,1)$ and $M\in[1,\infty)$
such that
\[
\tilde{P} V(x,w) \le\lambda V(x,w) \qquad\mbox{for all $w\in(0,\bar{w}]$, $|x|\ge M$.}
\]
\end{lemma}
%

\begin{pf} 
We may write
\begin{eqnarray*}
\frac{\tilde{P} V(x,w)}{V(x,w)} &=& 1 + \iint_{A_{x,w}} \hat{a}_{x,w}(z,u)
Q_{x+z}(\ud u) q(\ud z)
\\
&&{} +\iint_{R_{x,w}} \hat{b}_{x,w}(z,u)
Q_{x+z}(\ud u) q(\ud z),
\end{eqnarray*}
where
%
\begin{eqnarray}
\hat{a}_{x,w}(z,u) &\defeq &\biggl( \frac{\pi(x)}{\pi(x+z)}
\biggr)^{\eta} \frac{u^{-\alpha} \vee
u^{\beta}}{w^{-\alpha} \vee
w^{\beta}}-1, \label{eq:a-def-x-large}
\\
\hat{b}_{x,w}(z,u) &\defeq& \biggl(\frac{\pi(x+z)}{\pi(x)}
\biggr)^{1-\eta} \frac{u}{w} \biggl[ \frac{u^{-\alpha} \vee
u^{\beta}}{w^{-\alpha} \vee
w^{\beta}} - \biggl(
\frac{\pi(x+z)}{\pi(x)} \biggr)^{\eta} \biggr]. \label{eq:b-def-x-large}
\end{eqnarray}

Fix a constant $c>1$ and define the following subsets:
$\bar{A}_x \defeq
\{z\given
\frac{\pi(x+z)}{\pi(x)} \ge c \}$ and $\bar{R}_x \defeq\{z\given
\frac{\pi(x+z)}{\pi(x)} \le\frac{1}{c} \}$, and the annulus
between these two sets as $D_x \defeq(\bar{A}_x\cup
\bar{R}_x)^\complement= \{z\given\frac{1}{c} <
\frac{\pi(x+z)}{\pi(x)}< c\}$. Compute
%
\begin{eqnarray}\label{eq:bounded-d-a}
&&\int_{D_x} \int_{(z,u)\in A_{x,w}}
\hat{a}_{x,w}(z,u) Q_{x+z}(\ud u) q(\ud z)
\nonumber
\\[-8pt]
\\[-8pt]
\nonumber
&&\qquad\le\frac{c^\eta}{w^{-\alpha}\vee w^{\beta}} \int_{D_x} \int u^{-\alpha}\vee
u^{\beta} Q_{x+z}(\ud u) q(\ud z) \le M_W
c^\eta q(D_x)
\end{eqnarray}
and
%
\begin{eqnarray} \label{eq:bounded-d-r}\qquad
&&\int_{D_x} \int_{(z,u)\in R_{x,w}}
\hat{b}_{x,w}(z,u) Q_{x+z}(\ud u) q(\ud z)
\nonumber
\\[-8pt]
\\[-8pt]
\nonumber
&&\qquad\le c^{1-\eta} \int_{D_x} \int_{u< cw}
\biggl(\frac{u}{w} \biggr) \frac{u^{-\alpha}\vee u^{\beta}}{w^{-\alpha}\vee
w^{\beta}} Q_{x+z}(\ud u)q(\ud
z) \le M_W c^{2-\eta}q(D_x).
\end{eqnarray}

Let then $\gamma\in(\eta,\alpha\wedge1)$
such that $\gamma+\beta\le\beta'$
and observe that
$ (\frac{\pi(x+z)}{\pi(x)}\frac{u}{w} )^{1-\gamma}\le1$ on
$R_{x,w}$, and thereby
%
\begin{eqnarray}
\label{eq:bounded-a-a}&& \int_{\bar{R}_x}\int_{(z,u)\in R_{x,w}}
\hat{b}_{x,w}(z,u) Q_{x+z}(\ud u) q(\ud z)
\nonumber\\
&&\qquad\le\int_{\bar{R}_x} \int_{(z,u)\in R_{x,w}} \biggl(
\frac{\pi(x+z)}{\pi(x)} \biggr)^{\gamma-\eta} \frac{u^{\gamma-\alpha} \vee
u^{\gamma+\beta}}{w^{\gamma-\alpha} \vee
w^{\gamma+\beta}} Q_{x+z}(\ud
u) q(\ud z)
\\
&&\qquad\le M_W c^{-(\gamma-\eta)}.
\nonumber
\end{eqnarray}
Similarly, observe that
$ (\frac{\pi(x)}{\pi(x+z)}\frac{w}{u} )^{\gamma}\le1$ on
$A_{x,w}$ and so
%
\begin{eqnarray} \label{eq:bounded-r-a}
&&\int_{\bar{R}_x} \int_{(z,u)\in A_{x,w}}
\hat{a}_{x,w}(z,u) Q_{x+z}(\ud u) q(\ud z)
\nonumber\\
&&\qquad\le \frac{c^{-(\gamma-\eta)}}{w^{\gamma-\alpha} \vee
w^{\gamma+\beta}} \int_{\bar{R}_x} \int_{(z,u)\in A_{x,w}}
u^{\gamma-\alpha}\vee u^{\gamma+\beta} Q_{x+z}(\ud u) q(\ud z)
\\
&&\qquad\le M_W c^{-(\gamma-\eta)}.
\nonumber
\end{eqnarray}

It holds that $1\le (\frac{\pi(x)}{\pi(x+z)}\frac{w}{u} )$
on $R_{x,w}$, so we have
%
\begin{eqnarray}\label{eq:bounded-a-r}
&&\int_{\bar{A}_x} \int_{(z,u)\in R_{x,w}}
\hat{b}_{x,w}(z,u) Q_{x+z}(\ud u) q(\ud z)
\nonumber\\
&&\qquad \le\frac{1}{w^{-\alpha} \vee
w^{\beta}} \int_{\bar{A}_x} \int_{(z,u)\in R_{x,w}}
\biggl(\frac{\pi(x+z)}{\pi(x)} \biggr)^{-\eta} u^{-\alpha}\vee
u^{\beta} Q_{x+z}(\ud u) q(\ud z)
\\
&&\qquad \le M_W c^{-\eta}.
\nonumber
\end{eqnarray}

We are left with the term that will yield the geometric drift
when $|x|$ is large,
\begin{eqnarray*}
&&\int_{\bar{A}_x} \int_{(z,u)\in A_{x,w}}
\hat{a}_{x,w}(z,u) Q_{x+z}(\ud u) q(\ud z)
\\
&&\qquad\le\frac{M_W c^{-\eta} }{w^{-\alpha}\vee w^\beta} - \int_{\bar{A}_x} q(\ud z) \int
_{\{u: (z,u)\in A_{x,w}\}} Q_{x+z}(\ud u)
\\
&&\qquad\le M_W
c^{-\eta} - q(\bar{A}_x) \biggl(1- M_W \biggl(
\frac{w}{c} \biggr)^{\alpha'} \biggr),
\end{eqnarray*}
by Lemma~\ref{lem:acc-prob}(ii).
Lemma~\ref{lem:containment-and-delta-acc}(iii) implies
that $\delta\defeq\liminf_{|x|\to\infty} q(\bar{A}_x)>0$.

Let $\delta'\in(0,\delta)$ and
fix $\varepsilon>0$ sufficiently small so that
$6\varepsilon- \delta(1-\varepsilon)^2\le-\delta'$, and let $c>1$
be sufficiently
large so that $M_W c^{-\eta}\le\varepsilon$ and $M_W
 (\frac{\bar{w}}{c} )^{\alpha'}\le\varepsilon$, and also that
all \eqref{eq:bounded-a-a}, \eqref{eq:bounded-r-a}
and \eqref{eq:bounded-a-r} are bounded by $\varepsilon$.
Condition \ref{a:super-exp-regular} implies that
$\limsup_{|x|\to\infty} q(D_x) = 0$,
and therefore there exists
$M=M(c,\varepsilon)>0$ such that
$\eqref{eq:bounded-d-a}+\eqref{eq:bounded-d-r}\le\varepsilon$
for all $|x|\ge M$.
By possibly increasing the bound
$M$ to ensure that $q(\bar{A}_x)\geq\delta(1-\varepsilon)$,
we have that the claim holds for all $|x|\ge
M$ with the constant $\lambda=1-\delta'$.
\end{pf}

We complete the results above by considering in particular very small
values of~$w$.

%
\begin{lemma}
\label{lem:bounded-in-compacts} 
Suppose the conditions of Theorem~\ref{thm:rwm-bounded-moments} hold,
and let $\bar{w},M\in[1,\infty)$. Then,
there exist constants $\underline{w}\in(0,1)$, $\lambda\in(0,1)$ and
$b\in[1,\infty)$ such that
%
\begin{eqnarray}
\tilde{P} V(x,w) &\le& b,\qquad\mbox{for $|x|\le M$ and $w\in[\underline{w},
\bar{w}]$},\label{eq:bounded-on-c}
\\
\tilde{P} V(x,w) &\le&\lambda V(x,w),\qquad\mbox{for $x\in\mathsf{X}$ and $w\in(0,
\underline{w}]$.} \label{eq:drift-with-small-w}
\end{eqnarray}
\end{lemma}
%

\begin{pf} 
From the proof of Lemma~\ref{lem:drift-w-bounded}, we have
\[
\frac{\tilde{P} V(x,w)}{V(x,w)} \le1
- \biggl(\iint_{A_{x,w}} Q_{x+z}(\ud u)
q(\ud z) \biggr) + \tilde{a}_{x,w} + \tilde{b}_{x,w},
\]
where
\begin{eqnarray*}
\tilde{a}_{x,w}&\defeq&\iint_{A_{x,w}} \biggl( \frac{\pi(x)}{\pi(x+z)}
\biggr)^{\eta} \frac{u^{-\alpha} \vee
u^{\beta}}{w^{-\alpha} \vee
w^{\beta}} Q_{x+z}(\ud u) q(\ud z),
\\
\tilde{b}_{x,w}&\defeq&\iint_{R_{x,w}} \biggl(\frac{\pi(x+z)}{\pi
(x)}
\biggr)^{1-\eta} \frac{u}{w} \frac{u^{-\alpha} \vee
u^{\beta}}{w^{-\alpha} \vee
w^{\beta}} Q_{x+z}(\ud
u) q(\ud z).
\end{eqnarray*}
Because $ (\frac{\pi(x)}{\pi(x+z)}\frac{w}{u} )^{\eta}\le1$
on $A_{x,w}$ and $ (\frac{\pi(x+z)}{\pi(x)}\frac{u}{w}
)^{1-\eta
}\le1$
on $R_{x,w}$,
\[
\tilde{a}_{x,w}+\tilde{b}_{x,w} \le\iint \frac{u^{\eta-\alpha}\vee u^{\eta+\beta}}{
w^{\eta-\alpha}\vee w^{\beta+\alpha}}
Q_{x+z}(\ud u) q(\ud z) \le\frac{M_W}{w^{\eta-\alpha}\vee w^{\beta+\alpha}}.
\]
This is enough to show that $\tilde{P}V(x,w) \le(1+M_W) V(x,w)$ for
all $(x,w)\in\mathsf{X}\times\mathsf{W}$.
Because $V$ is bounded on $\{|x|\le M,
w\in[\underline{w},\bar{w}]\}$,
this implies the existence of
$b=b(\bar{w},\underline{w},M)<\infty$ such that \eqref{eq:bounded-on-c}
holds.

Consider then \eqref{eq:drift-with-small-w}. Let $\delta>0$ be small
enough so that $\inf_{x\in\mathsf{X}} q(A_x^\delta)\ge\varepsilon>0$,
where $A_x^\delta\defeq\{z\given\frac{\pi(x+z)}{\pi(x)}\ge\delta
\}$.
Then
\begin{eqnarray*}
\iint_{A_{x,w}} Q_{x+z}(\ud u) q(\ud z) &\ge&\int
_{A_x^\delta} q(\ud z) \int_{\{u\given(z,u)\in A_{x,w}\}}
Q_{x+z}(\ud u)
\\
&\ge&\int_{A_x^\delta} q(\ud z) \biggl(1-M_W \biggl(
\frac{w}{\delta} \biggr)^{\alpha'} \biggr) \ge\frac{\varepsilon}{2}
\end{eqnarray*}
for $w\in(0,\underline{w}]$ if $\underline{w}$ is small enough.
We may further decrease $\underline{w}$ to ensure that
$\tilde{a}_{x,w}+\tilde{b}_{x,w}\le\varepsilon/4$ for all $w\in
(0,\bar{w}]$
and conclude \eqref{eq:drift-with-small-w}
with $\lambda= 1-\varepsilon/4$.
\end{pf}
%


\subsection{Nonuniform moment bounds}
\label{sec:nonuniform-moment-bounds} 

We replace the uniform moments in Condition \ref{cond:w-moments}
here with the following assumption, which allows the moments of the
distributions $\{Q_x\}_{x\in\mathsf{X}}$ to grow in the tails of $\pi$.

%
\begin{condition}
\label{cond:w-moments-nonuniform} 
Let $\hat{w}\dvtx\mathsf{X}\to[1,\infty)$
be a function bounded on compact sets and tending to infinity as
$|x|\to\infty$. Let $\psi\dvtx(0,\infty)\to[1,\infty)$ be a
nonincreasing function such that $\psi(t)\to\infty$ as $t\to0$, and
define $g(x)\defeq\psi(\pi(x))$.

\begin{longlist}[(iii)]
\item[(i)]
There exist constants $\alpha'>0$ and $\beta'>1$ such
that
\[
\esssup_{x\in\mathsf{X}} g^{-1}(x) \int u^{-\alpha'}\vee
u^{\beta'} Q_x(\ud u) \le1,
\]
where the essential supremum is taken with respect to the Lebesgue
measure.
\item[(ii)]
There exist constants $\xi_w\in(0,\beta'-1)$
and $\xi_\pi\in(0,\beta'-1-\xi_w)$,
%
\begin{equation}
\sup_{x\in\mathsf{X}} \frac{g(x)}{\hat{w}^{\xi_\pi}(x)} \sup_{z\in R_x}
\biggl[ \biggl(\frac{\pi(x+z)}{\pi(x)} \biggr)^{\xi_\pi} \frac{g(x+z)}{g(x)}
\biggr] <\infty, \label{eq:g-pi-balance}
\end{equation}
where $R_x \defeq\{z\given\frac{\pi(x+z)}{\pi(x)}< 1\}$ is the set of
possible rejection for the marginal random-walk Metropolis algorithm.
\item[(iii)]
For any constant $b>1$, one must have
%
\begin{equation}
\sup_{x\in\mathsf{X}} \frac{M_W (b(|x|\vee1) )}{\hat{w}^{\xi_w}(x)} < \infty, \label{eq:M-w-balance}
\end{equation}
where $M_W\dvtx(0,\infty)\to(0,\infty)$ is defined as follows:
\[
M_W(r)\defeq\esssup_{|x|\le r} \int u^{-\alpha'}\vee
u^{\beta'} Q_x(\ud u) \le\esssup_{|x|\le r} g(x),
\]
where the essential supremum is taken with respect to the Lebesgue
measure.
\end{longlist}
\end{condition}
%

The assumptions in Condition \ref{cond:w-moments-nonuniform} may appear
rather implicit
and technical at first. However they, together with additional assumptions
required in Theorem~\ref{thm:rwm-unbounded-moments} below, are implied
by the more meaningful assumptions in Condition \ref
{cond:strong-super-exp-regular} and
Corollary~\ref{cor:strongly-super-exp},
whose proof may help the reader gain some intuition.

%
\begin{theorem}
\label{thm:rwm-unbounded-moments} 
Suppose $\tilde{P}$ is a pseudo-marginal kernel corresponding to a random
walk Metropolis with invariant density $\pi$ and increment
proposal density $q$ satisfying
Condition \ref{a:super-exp-regular}.
Suppose Condition \ref{cond:w-moments-nonuniform} holds
with some $\alpha'>0$ and
$\beta'>1$. Define $V\dvtx\mathsf{X}\times\mathsf{W}\to[1,\infty
)$ as in
\eqref{eq:def-drift-function}, where the constant exponents
satisfy
\[
\eta\in \bigl(0,\alpha'\wedge\bigl(\beta'-1-
\xi_w\bigr)\wedge(1-\xi_\pi ) \bigr),\qquad \alpha\in\bigl(\eta,
\alpha'\bigr], \beta\in\bigl(1+\xi_w-\eta,
\beta'-\eta\bigr)
\]
and $\eta\le(\beta'-\beta)\wedge1-\xi_\pi$.

Furthermore, suppose that there
exists a function $c\dvtx\mathsf{X}\to[1,\infty)$ bounded on
compact sets
such that
$\limsup_{|x|\to\infty} c(x)e^{-x} < \infty$ and
%
\begin{equation}
\limsup_{|x|\to\infty} \frac{\hat{w}^{\xi_\pi}(x)}{c^{\xi_c}(x)} = 0 \qquad
\mbox{where }
\xi_c \in\bigl(0, \bigl[\bigl(\beta'-\beta\bigr)\wedge
\alpha\wedge1\bigr] - \eta- \xi _\pi\bigr), \label{eq:c-bigger-w}
\end{equation}
and that for any constant $b\in[1,\infty)$
%
\begin{equation}
\limsup_{|x|\to\infty} M_W\bigl(b|x|\bigr) \max \biggl\{
q(D_x), \frac{1}{c^{\eta}(x)}, \biggl(\frac{\hat{w}(x)}{c(x)}
\biggr)^{\alpha'} \biggr\} = 0, \label{eq:M-vanish}
\end{equation}
where $D_x \defeq \{z\given\frac{1}{c(x)} \le\frac{\pi
(x+z)}{\pi(x)}
\le c(x)  \}$.

Then, there exist constants $\bar{w}, M,b\in[1,\infty)$,
$\underline{w}\in(0,1]$ and $\delta_V>0$ such that the polynomial
drift inequality \eqref{eq:def-drift-condition} holds.
Furthermore, the constants depend only on those of the marginal algorithm,
the quantities $\alpha', \beta', \xi_w, \xi_\pi$, $\psi$,
$\hat{w}$ involved in Condition \ref{cond:w-moments-nonuniform},
including the upper bounds in \eqref{eq:g-pi-balance} and
\eqref{eq:M-w-balance} (as a function of $b$),
the chosen $\eta$, $\alpha$, $\beta$, $c$ and $\xi_c$, and the upper
bounds \eqref{eq:c-bigger-w} and \eqref{eq:M-vanish}.
\end{theorem}
%

\begin{pf} 
The proof follows by applying Lemma~\ref{lem:drift-w-large-nonuniform}
below and
then Lem\-ma~\ref{lem:drift-w-bounded-nonuniform} with $c_w$ from
Lemma~\ref{lem:drift-w-large-nonuniform}, similarly to
the proof of Theorem~\ref{thm:rwm-bounded-moments}
by setting $\bar{w} \defeq\sup_{|x|\le M} \bar{w}(x)$,
and observing that $V$ is bounded on $C$.
The dependence on the various quantities is clear from the proofs of
Lemmas \ref{lem:drift-w-large-nonuniform} and
\ref{lem:drift-w-bounded-nonuniform}.
\end{pf}

Before proving Lemmas \ref{lem:drift-w-large-nonuniform} and
\ref{lem:drift-w-bounded-nonuniform}, we give
sufficient conditions to establish the conditions of
Theorem~\ref{thm:rwm-unbounded-moments}.

%
\begin{condition}
\label{cond:strong-super-exp-regular} 
Suppose Condition
\ref{a:super-exp-regular} holds and additionally there exists a
constant $\rho>1$ such that
\[
\lim_{|x|\to\infty} \frac{x}{|x|^{\rho}} \cdot\nabla \log\pi(x) = -
\infty.
\]
Moreover, the increment proposal density $q$ satisfies
$q(x)\le\bar{q}(|x|)$ for some bounded differentiable
nonincreasing function $\bar{q}\dvtx[0,\infty)\to[0,\infty)$
such that $\int_\mathsf{X} \bar{q}(|x|) \,\ud x<\infty$.
\end{condition}
%

%
\begin{corollary}
\label{cor:strongly-super-exp} 
Suppose Condition \ref{cond:strong-super-exp-regular} is satisfied,
and that
%
\begin{equation}
\int u^{-\alpha'}\vee u^{\beta'} Q_x(\ud u) \le c \bigl(1
\vee|x| \bigr)^{\rho'} \label{eq:simple-moment-bound}
\end{equation}
with some constants $c<\infty$ and $\rho'\in[0,\rho-1)$.
Then, for any
\[
\eta\in \bigl(0,\alpha'\wedge\bigl(\beta'-1\bigr)
\wedge1 \bigr), \qquad\alpha\in\bigl(\eta,\alpha'\bigr], \beta\in\bigl(1-\eta,
\beta'-\eta\bigr)
\]
and $V$ defined in \eqref{eq:def-drift-function},
the drift inequality \eqref{eq:def-drift-condition} holds,
with constants $\bar{w}, M,b\in[1,\infty)$,
$\underline{w}\in(0,1]$, and $\delta_V>0$ only depending on the
marginal algorithm and $\alpha',\beta',c,\rho'$ in
\eqref{eq:simple-moment-bound} and the chosen $\alpha,\beta$, and
$\eta$.
\end{corollary}
%

\begin{pf} 
Choose the constants $\xi_w$ and $\xi_\pi$ sufficiently small so
that the conditions on $\eta$, $\alpha$, and $\beta$ in Theorem~\ref{thm:rwm-unbounded-moments} are satisfied.

Fix a unit vector $u\in\R^d$, and define the function
$\hat{\psi}\dvtx\R_+\to[1,\infty)$ such that
\[
\hat{\psi}\bigl(\pi(ru)\bigr) = \cases{r,&\quad $r\ge R_0,$ \vspace*{2pt}
\cr
R_0,& \quad $r\in[0,R_0)$,}
\]
where $R_0\in[1,\infty)$; this is always possible
because the
function $r\mapsto\pi(ru)$ is bounded away from zero on compact sets
and monotone decreasing on the tail.

Define then $g(x) = c_g\hat{\psi}^{\rho'}(\pi(x))$, where the value of
the constant $c_g\ge1$ will be fixed later.
In order to guarantee that
Condition \ref{cond:w-moments-nonuniform}(i) is satisfied for sufficiently large $c_g$,
it is sufficient to show that
%
\begin{equation}
\limsup_{|x|\to\infty} g^{-1}(x) |x|^{\rho'}<\infty.
\label{eq:g-dominates-moments}
\end{equation}
Due to Lemma~\ref{lem:annulus-containment}
in Appendix \ref{171717171717-lemmas}, if $|x|$ is sufficiently
large, then $g(x) = g(\zeta_x|x|u)$ for some $\zeta_x\in[b^{-1},b]$, where
$b\in[1,\infty)$ is a constant. Therefore,
$g^{-1}(x) \le(b^{-1}|x|)^{-\rho'}$,
implying \eqref{eq:g-dominates-moments}.

Define then $\hat{w}(x) \defeq g^{\zeta_w}(x)$, where
$\zeta_w = \xi_\pi^{-1}\vee\xi_w^{-1}\in(1,\infty)$.
It is easy to check similarly to \eqref{eq:g-dominates-moments} that
\[
\sup_{x\in\mathsf{X}} \frac{g(x)}{\hat{w}^{\xi_\pi}(x)} + \frac{M_W (b(|x|\vee1) )}{\hat{w}^{\xi_w}(x)} \le1 +
\sup_{x\in\mathsf{X}} \frac{c' (b|x|)^{\rho'}}{\hat{w}^{\xi_w}(x)}< \infty.
\]
It is also easy to check that
\[
\sup_{z\in R_x} \biggl[ \biggl(\frac{\pi(x+z)}{\pi(x)}
\biggr)^{\xi_\pi} \frac{g(x+z)}{g(x)} \biggr] = \sup_{z\in R_x}
\biggl[ \biggl(\frac{\pi(x+z)}{\pi(x)} \biggr)^{\xi_\pi} \biggl(
\frac{\hat{\psi}(\pi(x+z))}{\hat{\psi}(\pi(x))} \biggr)^{\rho'} \biggr]
\]
is uniformly bounded in $x\in\mathsf{X}$. This is because it is
sufficient to check the condition in
the tails along a ray, that is, only for
$z=r|x|$, $r\ge1$.
We conclude about the existence of a constant $c_g\in[1,\infty)$ such that
Condition \ref{cond:w-moments-nonuniform} holds.

Choose $\varepsilon_c\in(0,\rho-1-\rho')$, and let $c(x) =
\exp(|x|^{\varepsilon_c})$.
It is easy to check that there exists $\xi_c$ such that
\eqref{eq:c-bigger-w} and \eqref{eq:M-vanish} hold, using
Lemma~\ref{lem:q-d-x} in Appendix \ref{171717171717-lemmas} to
estimate $q(D_x)$.\vspace*{-1pt}
\end{pf}
%

We start by establishing a polynomial drift when $w$ is
large.

%
\begin{lemma}
\label{lem:drift-w-large-nonuniform} 
Suppose the conditions of Theorem~\ref{thm:rwm-unbounded-moments} hold.
Then there exist constants $c_w\in[1,\infty)$ and $\delta_V>0$ such that
letting $\bar{w}(x)\defeq c_w \hat{w}(x)$,
\[
\tilde{P} V(x,w) \le V(x,w) - \delta_V V^{{(\beta-1)}/{\beta}}(x,w)
\qquad\hspace*{-1pt}\mbox{for all $x\in\R^d$ and $w\in\bigl[\bar{w}(x),\infty\bigr)$.}
\]
\end{lemma}
%

\begin{pf} 
We may write
\begin{eqnarray*}
\frac{\tilde{P} V(x,w)}{V(x,w)} &=& \iint_{A_{x,w}} a_{x,w}(z,u)
Q_{x+z}(\ud u)q(\ud z)
\\
&&{}+ \iint_{R_{x,w}} b_{x,w}(z,u)
Q_{x+z}(\ud u)q(\ud z),
\end{eqnarray*}
where $a_{x,w}$ and $b_{x,w}$ are defined in \eqref{eq:a-w-large} and
\eqref{eq:b-w-large}, respectively.

In what follows, for any $\nu>0$, we will denote by
$b_\nu\in(0,\infty)$ a constant
chosen so that for all $x\in\mathsf{X}$,
$ \{x+z\given\frac{\pi(x+z)}{\pi(x)}\ge\nu \}\subset
B (0,b_\nu(|x|\vee1) )$;
see Lemma~\ref{lem:containment-and-delta-acc}(i)
in Appendix \ref{171717171717-lemmas}. We also denote by $c\in[1,\infty)$
a constant whose value may change upon each appearance.

For the first integral, note that on $A_{x,w}$,
$1\le (\frac{\pi(x+z)}{\pi(x)}\frac{u}{w} )^{\eta}$, so
denoting $\delta\defeq\eta+\beta-1-\xi_w>0$, we have
for $w\ge\hat{w}(x)$,
\begin{eqnarray*}
&&\iint_{A_{x,w}\cap A_x} a_{x,w}(z,u) Q_{x+z}(\ud u)q(\ud z) \\
&&\qquad
\le\iint_{A_{x,w}\cap A_x} \frac{u^{\eta-\alpha}\vee
u^{\eta+\beta}}{w^{\eta+\beta}} Q_{x+z}(\ud u)q(\ud z)
\\
&&\qquad\le\frac{1}{w^{1+\delta}} \biggl(\frac{M_W (b_1(|x|\vee1) )}{\hat{w}^{\xi
_w}(x)} \biggr)\le
\frac{c}{w^{1+\delta}},
\end{eqnarray*}
by Condition \ref{cond:w-moments-nonuniform}(iii).
For the second one, let $\gamma\in(\eta+\xi_\pi,\beta'-\beta]$,
$\gamma<1$, and observe that $1\le
 (\frac{\pi(x+z)}{\pi(x)}\frac{u}{w} )^{\gamma}$ on
$A_{x,w}$, implying
that with $\delta'\defeq\gamma+\beta-1-\xi_\pi>0$
\begin{eqnarray*}
&&\iint_{A_{x,w}\cap R_x} a_{x,w}(z,u) Q_{x+z}(\ud u)q(\ud z)
\\
&&\qquad\le \int_{R_x} \biggl(\frac{\pi(x+z)}{\pi(x)}
\biggr)^{\gamma-\eta} \frac{u^{\gamma-\alpha} \vee u^{\gamma+\beta}}{w^{\gamma+\beta}} Q_{x+z}(\ud u) q(\ud z)
\\
&&\qquad\le\frac{1}{w^{1+\delta'}} \int_{R_x} \biggl[ \biggl(
\frac{\pi(x+z)}{\pi(x)} \biggr)^{\xi_\pi} \frac{g(x+z)}{g(x)} \biggr]
\frac{g(x)}{\hat{w}^{\xi_\pi}(x)} q(\ud z) \le\frac{c}{w^{1+\delta'}},
\end{eqnarray*}
whenever $w\ge\hat{w}(x)$, by Condition
\ref{cond:w-moments-nonuniform}(i) and
(ii).
Similarly, because\break
$ (\frac{\pi(x+z)}{\pi(x)}\frac{u}{w} )^{1-\gamma}\le1$ on
$R_{x,w}$ we have for $w\ge\hat{w}(x)$,
\begin{eqnarray*}
&&\iint_{R_{x,w}\cap R_x} \biggl(\frac{\pi(x+z)}{\pi(x)} \biggr)^{1-\eta}
\frac{u^{1-\alpha}\vee u^{1+\beta}}{w^{1+\beta}} Q_{x+z}(\ud u) q(\ud z)
\\
&&\qquad\le\frac{1}{w^{1+\delta'}} \int_{R_x} \biggl[ \biggl(
\frac{\pi(x+z)}{\pi(x)} \biggr)^{\xi_\pi} \frac{g(x+z)}{g(x)} \biggr]
\frac{g(x)}{\hat{w}^{\xi_\pi}(x)} q(\ud z) \\
&&\qquad\le\frac{c}{w^{1+\delta'}},
\end{eqnarray*}
and similarly, because
$ (\frac{\pi(x+z)}{\pi(x)}\frac{u}{w} )^{1-\eta}\le1$,
\begin{eqnarray*}
&&\iint_{R_{x,w}\cap A_x} \biggl(\frac{\pi(x+z)}{\pi(x)} \biggr)^{1-\eta}
\frac{u^{1-\alpha}\vee u^{1+\beta}}{w^{1+\beta}} Q_{x+z}(\ud u) q(\ud z)
\\
&&\qquad\le \frac{1}{w^{1+\delta}} \biggl(
\frac{M_W (b_1(|x|\vee1) )}{\hat{w}^{\xi_w}(x)} \biggr)\le \frac{c}{w^{1+\delta}}.
\end{eqnarray*}

As in the proof of Lemma~\ref{lem:drift-w-large}, we may apply
Lemma~\ref{lem:acc-prob}(i) to obtain
\begin{eqnarray*}
&&\iint_{R_{x,w}}  \biggl(1-\frac{\pi(x+z)}{\pi(x)}\frac{u}{w} \biggr)
Q_{x+z}(\ud u) q(\ud z)
\\
&&\qquad\le1 - \frac{\nu}{w}\int_{\{z\given({\pi(x+z)}/{\pi(x)})\ge
\nu\}} \biggl(1-
\frac{1}{w^{\beta'-1}} \int u^{\beta'} Q_{x+z}(\ud u) \biggr) q(\ud z)
\\
&&\qquad\le1 - \frac{\nu}{w}\int_{\{z\given({\pi(x+z)}/{\pi(x)})\ge
\nu\}}q(\ud z) \biggl(1-
\frac{1}{w^{\beta'-1-\xi_w}} \biggl(\frac{M_W (b_\nu(|x|\vee1) )}
{\hat{w}^{\xi
_w}(x)} \biggr) \biggr)
\\
&&\qquad\le1 - \frac{\nu}{w}\int_{\{z\given({\pi(x+z)}/{\pi(x)})\ge
\nu\}}q(\ud z) \biggl(1-
\frac{c}{w^{\beta'-1-\xi_w}} \biggr), 
\end{eqnarray*}
where we may choose $\nu\in(0,1)$
such that
$\inf_{x\in\mathsf{X}} q (z\given\frac{\pi(x+z)}{\pi(x)}\ge
\nu )>0$;
Lem\-ma~\ref{lem:containment-and-delta-acc}(ii)
ensures the existence of such a $\nu$.

The terms of the order $w^{-(1+\delta)}$ or $w^{-(1+\delta')}$
vanish faster than $w^{-1}$ as $w$ increases.
Consequently, we can choose $c_w\in[1,\infty)$ sufficiently large so
that there exists a $\nu'>0$ such that
for all $x\in\mathsf{X}$ and $w\ge\bar{w}(x)$,
\begin{eqnarray*}
\tilde{P}V(x,w) &\le& \biggl(1- \frac{\nu'}{w} \biggr)V(x,w)
\\
&=& V(x,w) - \delta_V V^{\kappa}(x,w) \bigl(c_\pi^\eta
\pi^{-\eta}(x) \bigr)^{1-\kappa} \le V(x,w) - \delta_V
V^{\kappa}(x,w),
\end{eqnarray*}
where $\kappa= \frac{\beta-1}{\beta}\in(0,1)$.
\end{pf}
%

Our last lemma concentrates on the cases where
either $|x|$ is large and $w$ bounded, or $w$ is small.

%
\begin{lemma}
\label{lem:drift-w-bounded-nonuniform} 
Assume the conditions of Theorem~\ref{thm:rwm-unbounded-moments} hold
and let $\bar{w}(x)\defeq c_w
\hat{w}(x)$ for some constant $c_w\in[1,\infty)$.
Then, there exist constants $\lambda\in(0,1)$, $\underline{w}\in(0,1)$,
$M\in[1,\infty)$, and $c_V\in[1,\infty)$ such that
%
\begin{eqnarray}
\tilde{P} V(x,w) &\le&\lambda V(x,w) \qquad
\mbox{for $|x|\ge M, w\in\bigl(\underline{w},
\bar{w}(x)\bigr]$}, \label{eq:drift-x-large}
\\
\tilde{P} V(x,w) &\le&\lambda V(x,w) \qquad \mbox{for $x\in\mathsf{X}, w\in(0,
\underline{w}]$}, \label{eq:drift-w-small}
\\
\tilde{P} V(x,w) &\le& c_V V(x,w) \qquad \mbox{for $(x,w)\in\mathsf{X}
\times\mathsf{W}$.} \label{eq:drift-crude}
\end{eqnarray}
\end{lemma}
%

\begin{pf} 
We may write
\begin{eqnarray*}
\frac{\tilde{P} V(x,w)}{V(x,w)} &=& 1 + \iint_{A_{x,w}} \hat{a}_{x,w}(z,u)
Q_{x+z}(\ud u) q(\ud z)
\\
& &{}+\iint_{R_{x,w}} \hat{b}_{x,w}(z,u)
Q_{x+z}(\ud u) q(\ud z),
\end{eqnarray*}
where $\hat{a}_{x,w}$ and $\hat{b}_{x,w}$ are given as in
\eqref{eq:a-def-x-large} and \eqref{eq:b-def-x-large}.

Define the subsets $\bar{A}_x \defeq
\{z\given
\frac{\pi(x+z)}{\pi(x)} \ge c(x) \}$, $\bar{R}_x \defeq\{z\given
\frac{\pi(x+z)}{\pi(x)} \le\frac{1}{c(x)} \}$ and
$D_x \defeq(\bar{A}_x\cup
\bar{R}_x)^\complement= \{z\given\frac{1}{c(x)} <
\frac{\pi(x+z)}{\pi(x)}< c(x)\}$. Lemma~\ref{lem:annulus-containment} in Appendix \ref{171717171717-lemmas}
implies the existence of $b_1\in[1,\infty)$ and $M_0\in[1,\infty)$
such that
$\bar{A}_x \cup D_x +x\subset B (0,b_1(|x|\vee1) )$ for all
$x\in\mathsf{X}$. We decompose the two sums above into sub-sums on
$\bar
{A}_x$ and $\bar{R}_x$,
with again an obvious abuse of notation.

Observe that
$1\le (\frac{\pi(x+z)}{\pi(x)} \frac{u}{w} )^{\eta}$ on $A_{x,w}$
and $ (\frac{\pi(x+z)}{\pi(x)}\frac{u}{w} )^{1-\eta}\le1$ on
$R_{x,w}$, implying
%
\begin{eqnarray} \label{eq:annulus-est-simple}
&&\iint_{D_x\cap A_{x,w}} \hat{a}_{x,w}(z,u) Q_{x+z}(\ud u) q(
\ud z)+ \iint_{D_x\cap R_{x,w}} \hat{b}_{x,w}(z,u) Q_{x+z}(\ud
u) q(\ud z)\nonumber
\\
&&\qquad\le\int_{D_x} \int \frac{u^{\eta-\alpha}\vee u^{\eta+\beta}}{w^{\eta-\alpha}\vee
w^{\eta+\beta}} Q_{x+z}(\ud
u)q(\ud z)
\\
&&\qquad\le\frac{M_W (b_1(|x|\vee1) ) q(D_x)}{w^{\eta-\alpha}\vee
w^{\eta+\beta}},
\nonumber
\end{eqnarray}
because $\eta\le(\beta'-\beta) \wedge\alpha$.

Let then $\gamma\defeq\eta+\xi_\pi+\xi_c < (\beta'-\beta)\wedge
\alpha
\wedge1$
and notice again that\break
$ (\frac{\pi(x+z)}{\pi(x)}\frac{u}{w} )^{1-\gamma}\le1$ on
$R_{x,w}$
and $ (\frac{\pi(x)}{\pi(x+z)}\frac{w}{u} )^{\gamma}\le1$ on
$A_{x,w}$. Therefore,
\begin{eqnarray*}
&&\iint_{\bar{R}_x\cap A_{x,w}} \hat{a}_{x,w}(z,u) Q_{x+z}(\ud u) q(
\ud z) + \iint_{\bar{R}_x\cap R_{x,w}} \hat{b}_{x,w}(z,u) Q_{x+z}(\ud
u) q(\ud z)
\\
&&\qquad \le\int_{\bar{R}_x} \biggl(\frac{\pi(x+z)}{\pi(x)}
\biggr)^{\gamma-\eta} \int \frac{u^{\gamma-\alpha} \vee
u^{\gamma+\beta}}{w^{\gamma-\alpha} \vee
w^{\gamma+\beta}} Q_{x+z}(\ud u) q(\ud z)
\\
&&\qquad \le\frac{1}{w^{\gamma-\alpha}\vee w^{\gamma+\beta}} \biggl(\frac{\hat{w}^{\xi_\pi}(x)}{c^{\xi_c}(x)} \biggr) \int
_{\bar{R}_x} \biggl[ \biggl(\frac{\pi(x+z)}{\pi(x)} \biggr)^{\xi_\pi}
\frac{g(x+z)}{g(x)} \biggr] \frac{g(x)}{\hat{w}^{\xi_\pi}(x)} q(\ud z),
\end{eqnarray*}
because $\frac{\pi(x+z)}{\pi(x)}\le c^{-1}(x)$ on $\bar{R}_x$.

It holds that $1\le
 (\frac{\pi(x)}{\pi(x+z)}\frac{w}{u} )$
on $R_{x,w}$, so we have
\begin{eqnarray*}
&&\int_{\bar{A}_x} \int_{(z,u)\in R_{x,w}}
\hat{b}_{x,w}(z,u) Q_{x+z}(\ud u) q(\ud z)
\\
&&\qquad\le\int_{\bar{A}_x} \biggl(\frac{\pi(x)}{\pi(x+z)}
\biggr)^{\eta} \int_{(z,u)\in R_{x,w}} \frac{u^{-\alpha}\vee
u^{\beta}}{w^{-\alpha}\vee w^{\beta}}
Q_{x+z}(\ud u) q(\ud z)
\\
&&\qquad\le\frac{M_W (b_1(|x|\vee1) ) c^{-\eta}(x)}{w^{-\alpha
}\vee
w^{\beta}}.
\end{eqnarray*}
Similarly,
\begin{eqnarray*}
&&\int_{\bar{A}_x} \int_{(z,u)\in A_{x,w}}
\hat{a}_{x,w}(z,u) Q_{x+z}(\ud u) q(\ud z)
\\
&&\qquad\le\frac{M_W (b_1(|x|\vee1) ) c^{-\eta}(x) }{w^{-\alpha
}\vee
w^{\beta}} - \int_{\bar{A}_x\cap A_{x,w}} Q_{x+z}(\ud u)
q(\ud z).
\end{eqnarray*}
Now, by Lemma~\ref{lem:acc-prob}(ii),
\begin{eqnarray*}
&&\int_{\bar{A}_x\cap A_{x,w}} Q_{x+z}(\ud u) q(\ud z)\\
&&\qquad\ge\int
\biggl(1 - \biggl(\frac{w}{c(x)} \biggr)^{\alpha'} \int u^{-\alpha'}
Q_{x+z}(\ud u) \biggr)q(\ud z)
\\
&&\qquad
\ge q(\bar{A}_x) \biggl[1- M_W \bigl(b_1\bigl(|x|
\vee1\bigr) \bigr) c_w^{\alpha'} \biggl(\frac{\hat{w}(x)}{c(x)}
\biggr)^{\alpha'} \biggr],
\end{eqnarray*}
for all $w\in(0,c_w \hat{w}(x)]$.

Lemma~\ref{lem:containment-and-delta-acc}(iii)
in Appendix \ref{171717171717-lemmas}
implies that
$\delta\defeq\liminf_{|x|\to\infty} q(\bar{A}_x)>0$.
Condition \ref{cond:w-moments-nonuniform} together
with \eqref{eq:c-bigger-w} and \eqref{eq:M-vanish}
imply
%
\begin{equation}
\limsup_{|x|\to\infty} \frac{\tilde{P} V(x,w)}{V(x,w)} \le1 - \delta,
\label{eq:drift-limit}
\end{equation}
and we may conclude \eqref{eq:drift-x-large},
by choosing any $\lambda\in(1-\delta,1)$
and finding a sufficiently large $M\in[1,\infty)$ such that the claim
holds.

Consider then \eqref{eq:drift-w-small} and assume $|x|\le M$.
It is easy to verify that \eqref{eq:drift-limit}
holds with some $\delta'>0$ when taking $\limsup_{w\to0+}$ in the
terms of the earlier decomposition.
Finally, it is easy to check that
\eqref{eq:drift-crude} holds for $|x|\le M$ similarly as
\eqref{eq:annulus-est-simple}, and the general case follows from
\eqref{eq:drift-x-large} and Lemma~\ref{lem:drift-w-large-nonuniform}.
\end{pf}
%



\section{Concluding remarks}
\label{sec:conclusion} 

Our convergence rate results in Sections~\ref{13131313131131311} and
\ref{151515151515}--\ref{171717171717} allow one to establish central limit
theorems. In the
case where the pseudo-marginal kernel is variance bounding, that is,
$\tilde{P}$ admits a spectral gap as discussed in Section~\ref
{13131313131131311}, the central limit theorem (CLT) holds for all
functions $f\dvtx\mathsf{X}\times\mathsf{W}\to\R$ such that
$\tilde{\pi}(f^2)<\infty$ \cite{roberts-rosenthal-variance-bounding},
Theorem~7. Specifically, we have for all
$g\dvtx\mathsf{X}\to\R$ with $\pi(g^2)<\infty$,
%
\begin{equation}
\frac{1}{\sqrt{n}} \sum_{k=0}^{n-1} \bigl[
g(\tilde{X}_k) - \pi(g) \bigr] \mathop{\longrightarrow}\limits^{n\to\infty} \mathcal{N}
\bigl(0,\var (g,\tilde {P}) \bigr)\qquad \mbox{in distribution}, \label{eq:clt}
\end{equation}
where $\var(g,\tilde{P})\in[0,\infty)$ is given in Definition~\ref{def:asvar}. It is possible to deduce upper bounds for the
asymptotic variance
$\var(g,\tilde{P})$. Namely,
Corollary~\ref{cor:autocorr-geom} relates $\var(g,\tilde{P})$ to
$\var(g,P)$, and from Lemma~\ref{lem:asvar-expressions},
\eqref{eq:asymptotic-variance},
\[
\var(g,P)\le \frac{1 + (1-\Gap(P))}{1-(1-\Gap(P))} \int e_{g-\pi(g),P}(\ud x) =
\frac{2-\Gap(P)}{\Gap(P)} \var_\pi(g),
\]
where $e_{g-\pi(g),P}$ is a positive measure on $[-1,1]$; see
Lemma~\ref{lem:asvar-expressions} in Appendix~\ref{sec:spectral}.
If the spectral gap of the marginal algorithm is not directly
accessible, it can be bounded by the drift constants; see
\cite{baxendale-bounds} and references therein, and also
\cite{latuszynski-miasojedow-niemiro}, Theorem~4.2(ii).

When $\tilde{P}$ is polynomially ergodic, the class of functions $g$
for which the CLT~\eqref{eq:clt} holds is related to the exponent in
the polynomial drift. For the convenience of the reader, we reformulate here
a result due to Jarner and Roberts~\cite{jarner-roberts}.

%
\begin{theorem}
\label{thm:petite-clt} 
Suppose $P$ is irreducible and aperiodic.
Assume
there exists
$V\dvtx\mathsf{X}\times\mathsf{W}\to[1,\infty)$, $\alpha\in[0,1)$,
$b\in[0,\infty)$, $c\in(0,\infty)$, a petite set
(e.g., \cite{jarner-hansen,meyn-tweedie})
$C\in\B(\mathsf{X})\times\B(\mathsf{W})$ such that
%
\begin{equation}
\tilde{P}V(x,w) \le V(x,w) - c V^\alpha(x,w) + b\mathbbm{I} \bigl\{
(x,w)\in C \bigr\}, \label{eq:polynomial-drift}
\end{equation}
and that there exists $\eta\in[1-\alpha,1]$ with
$\tilde{\pi}(V^{2\eta})<\infty$ and
\[
\sup_{(x,w)\in\mathsf{X}\times\mathsf{W}} \frac{|g(x)|}{V^{\alpha+\eta-1}(x,w)}<\infty,
\]
then $\var(g,\tilde{P})\in[0,\infty)$ and the CLT \eqref{eq:clt} holds.
\end{theorem}

%
Theorem~\ref{thm:petite-clt} is a restatement
of \cite{jarner-roberts}, Theorem~4.2, because the pseudo-marginal kernel
$\tilde{P}$ is also irreducible and aperiodic if the marginal kernel
$P$ is. The asymptotic variance can also be upper bounded
in the polynomial case; see \cite{andrieu-fort-vihola-subgeom} and
\cite{latuszynski-miasojedow-niemiro}, Theorem~5.2(ii) and Remark~5.3.
It is also possible to deduce nonasymptotic
mean square error bounds \cite{latuszynski-miasojedow-niemiro}.

Finally some of our results apply directly to extensions of
pseudo-marginal algorithms which directly make use of noisy estimates
of the marginal's acceptance ratio
\cite{karagiannis-andrieu,nicholls2012coupled}. However, despite some
similitudes and simplifications, the corresponding processes differ
fundamentally in that $(X_k)_{k \ge0}$ is a Markov chain
in this case (as opposed to the pseudo-marginal scenario), and we are
currently investigating these differences.


\begin{appendix}\label{app}
\section{Lemmas for Section \texorpdfstring{\protect\ref{12121212121212121}}{2}}\label{sec:spectral} 

In this section, $(\mathsf{X},\mathcal{B}(\mathsf{X}))$ is a generic
measurable space, and $\mu$ is a probability measure on $\mathsf{X}$.
We consider the Hilbert space
\[
L_0^2(\mathsf{X},\mu) \defeq\bigl\{f\dvtx\mathsf{X}\to\R
\given\mu(f)=0, \mu\bigl(f^2\bigr)<\infty\bigr\},
\]
equipped with the inner product
$ \langle f, g  \rangle_\mu\defeq\int_\mathsf{X} f(x)
g(x)\mu(\ud x)$. We denote the corresponding norm by $\|f\|_\mu\defeq
 \langle f, f  \rangle_\mu^{1/2}$ and the
operator norm for $A\dvtx L_0^2(\mathsf{X},\mu)\to L_0^2(\mathsf
{X},\mu)$ as
$\|A\| \defeq\sup\{\|A f\|_\mu\given\|f\|_\mu=1\}$.

%
\begin{lemma}
\label{lem:operator-calculus} 
Let $P_1$ and $P_2$ be two Markov kernels on space $\mathsf{X}$
reversible with respect to $\mu$, and define the family of interpolated
kernels $H_\beta\defeq P_1 +
\beta(P_2-P_1)$ for $\beta\in[0,1]$ also reversible with respect to
$\mu$. Then
\[
A_\lambda(\beta) \defeq(I-\lambda H_\beta)^{-1} (I+
\lambda H_\beta) = I + 2 \sum_{k=1}^\infty
\lambda^k H_\beta^k
\]
is a well-defined operator
on $L_0^2(\mathsf{X},\mu)$ for all
$\lambda\in[0,1)$ and $\beta\in[0,1]$ as well as the right-hand
derivatives, with limits taken with respect to the operator norm
\begin{eqnarray*}
A'_\lambda(\beta) &\defeq&\lim_{h\to0+}
h^{-1} \bigl(A_\lambda(\beta+h)-A_\lambda(\beta) \bigr)
\\
&=& 2\lambda(I-\lambda H_\beta)^{-1} (P_2-P_1)
( I - \lambda H_{\beta})^{-1},
\\
A''_\lambda(\beta) &\defeq&\lim
_{h\to0+} h^{-1} \bigl(A'_\lambda(
\beta+h)-A'_\lambda(\beta) \bigr)
\\
&= &2\lambda(I - \lambda H_\beta)^{-1} (P_2-P_1)A'_\lambda(
\beta),
\end{eqnarray*}
for all $\lambda\in[0,1)$ and $\beta\in[0,1)$.
\end{lemma}
%

\begin{pf} 
The expression for $A_\lambda(\beta)$ follows by the Neumann series
representation $(I-\lambda H_\beta)^{-1} = \sum_{k=0}^\infty(\lambda
H_\beta)^k$
which is well defined because\break $\| (\lambda H_\beta)^k \| \le\lambda^k$.
Let us check that $\beta\mapsto A_\lambda(\beta)$ is right differentiable
on $[0,1)$.
Write for any $h\in(0,1-\beta)$
\begin{eqnarray*}
A_\lambda(\beta+h)-A_\lambda(\beta) 
&=&
\lambda h (I-\lambda H_{\beta})^{-1}(P_2-P_1)
+ \Delta_{\lambda,\beta,h}(I + \lambda H_{\beta})
\\
&&{} + \lambda h\Delta_{\lambda,\beta,h}(P_2-P_1),
\end{eqnarray*}
where $\Delta_{\lambda,\beta,h} = (I-\lambda
H_{\beta+h})^{-1}-(I-\lambda H_\beta)^{-1}$.
The differentiability follows as soon as we show $\lim_{h\to0+}
h^{-1} (\Delta_{\lambda,\beta,h})$ exists.
By the Neumann series representation, it is sufficient to show that
$\lim_{h\to0+} h^{-1} (H_{\beta+h}^k - H_\beta^k )$ exists for all
$k\ge0$. The claim is trivial with $k=0$,
and the cases $k\ge1$ follow inductively by writing
\begin{eqnarray*}
H_{\beta+h}^k - H_\beta^k
&=& h H_\beta^{k-1} (P_2-P_1)
+ \bigl(H_{\beta+h}^{k-1}-H_\beta^{k-1}\bigr)
H_{\beta}
\\
&&{}+ h \bigl(H_{\beta+h}^{k-1}-H_\beta^{k-1}
\bigr) (P_2-P_1).
\end{eqnarray*}

Because $(I-\lambda H_\beta) A_\lambda(\beta) = I +
\lambda H_\beta$, we may write
\begin{eqnarray*}
\lambda h (P_2-P_1) 
&= (I-\lambda
H_{\beta+h}) \bigl(A_\lambda(\beta+h)-A_\lambda(\beta)
\bigr) - \lambda h (P_2-P_1) A_\lambda(\beta),
\end{eqnarray*}
from which, multiplying with $h^{-1}$ and taking limit as $h\to0+$, we
obtain
%
\begin{equation}
\lambda(P_2-P_1) = (I-\lambda H_\beta)
A'_\lambda(\beta) - \lambda(P_2-P_1)
A_\lambda(\beta). \label{eq:expr-aprime}
\end{equation}
The desired expression for $A'_\lambda(\beta)$ follows by observing that
$I+A_\lambda(\beta) = 2 (I-\lambda H_\beta)^{-1}$.
Consider then $A''_\lambda(\beta)$. From \eqref{eq:expr-aprime}, we obtain
\begin{eqnarray*}
&&(I-\lambda H_\beta)h^{-1} \bigl(A'_\lambda(
\beta+h)-A'_\lambda(\beta) \bigr)
\\
&&\qquad= \lambda(P_2-P_1) A'_\lambda(
\beta+h) + \lambda(P_2-P_1) h^{-1}
\bigl(A_\lambda(\beta+h)-A_\lambda(\beta ) \bigr).
\end{eqnarray*}
We conclude by taking limits as $h\to0+$.
\end{pf}
%

%
\begin{lemma}
\label{lem:asvar-expressions} 
Suppose $\Pi$ is a Markov kernel reversible with respect to $\mu$, and
$(X_n)_{n\ge0}$ is a Markov chain corresponding to the transition $\Pi
$ with
$X_0\sim\mu$. Then, for a function $f\in L_0^2(\mathsf{X},\mu)$
%
\begin{equation}
\var(f,\Pi) =\lim_{n\to\infty} \frac{1}{n} \E \Biggl(\sum
_{i=1}^n f(X_i)
\Biggr)^2 = \int\frac{1+x}{1-x} e_{f,\Pi}(\ud x) \in[0,
\infty], \label{eq:asymptotic-variance}
\end{equation}
where $e_{f,\Pi}$ is a positive measure on
$S\subset[-1,1]$ satisfying $e_{f,\Pi}(S)=\|f\|_\mu^2$.

For any $f\in L_0^2(\mathsf{X},\mu)$, whenever the series below is
convergent, then the following equality holds:
%
\begin{equation}
\var_\mu(f) + 2 \sum_{k=1}^\infty
\E\bigl[f(X_0)f(X_k)\bigr] = \var(f,\Pi) < \infty.
\label{eq:int-autocorr-eq-asvar}
\end{equation}
Moreover,
\[
\var_\lambda(f,\Pi) \defeq \bigl\langle f, (I-\lambda\Pi)^{-1}
(I+\lambda\Pi) f \bigr\rangle_{\mu}\in[0,\infty)
\]
is well defined for all $\lambda\in[0,1)$ and satisfies
$\lim_{\lambda\to1-} \var_\lambda(f,\Pi) = \var(f,\Pi)$ and
$ \langle f, (I-\lambda\Pi)^{-1} f  \rangle \ge0$.
\end{lemma}

%
The results in Lemma~\ref{lem:asvar-expressions} are well known;
a full proof is given in
\cite{andrieu-vihola-pseudo-arxiv}.


\section{Lemmas for Section \texorpdfstring{\protect\ref{13131313131131311}}{3}}
\label{13131313131131311-lemmas} 

We include the statement of
\cite{caracciolo-pelissetto-sokal}, Theorem A.2,
for the sake of
self-  containedness.

%
\begin{lemma}
\label{lem:operator-order} 
Let $A$ and $B$ be self-adjoint operators on a Hilbert space
$\mathcal{H}$ satisfying $0 \le \langle f, A f  \rangle
\le \langle f, B f  \rangle$ for
all $f\in\mathcal{H}$, and the inverses $A^{-1}$ and $B^{-1}$ exist.
Then $0 \le \langle f, B^{-1}f  \rangle \le \langle
f, A^{-1} f  \rangle$ for all
$f\in\mathcal{H}$.
\end{lemma}
%

%
\begin{lemma}
\label{lemma:gap-vs-accprob} 
Suppose $P$ is a Metropolis--Hastings kernel given in
\eqref{eq:marginal-kernel}, and $\rho(x)$ is
given in \eqref{eq:r-and-rho}. Then the spectral gap
of $P$ defined in \eqref{eq:spectral-gap} satisfies:
%
%
\begin{longlist}[(ii)]
\item[(i)]
for any set $A\in\mathcal{B}(\mathsf{X})$ with $\pi(A)\in(0,1)$,
\[
\Gap(P) \le \bigl(1-\pi(A) \bigr)^{-1} \Bigl(1-\inf
_{x\in A}\rho(x) \Bigr);
\]
\item[(ii)]
if $\pi$ does not have point masses, that is,
$\pi(\{x\})=0$ for all $x\in\mathsf{X}$, then
\[
\Gap(P) \le1- \rho(x)\qquad \mbox{for $\pi$-almost every $x\in\mathsf{X}$.}
\]
\end{longlist}
\end{lemma}
%

\begin{pf} 
We first check (i).
Denote $p=\P(A)\in(0,1)$ and
define $f(x) = a\mathbbm{I} \{x\in A \} - b\mathbbm
{I} \{x\notin A \}$ where
the constants $a,b\in(0,\infty)$ are chosen so that $\pi
(f)=ap-b(1-p)=0$ and
$\pi(f^2)=a^2 p + b^2 (1-p) = 1$.
We may compute
\begin{eqnarray*}
\mathcal{E}_P(f) &=& \frac{1}{2} \int\pi(\ud x) q(x,\ud y)
\min\bigl\{1,r(x,y)\bigr\} \bigl[f(x) - f(y)\bigr]^2
\\
&=& (a+b)^2 \int_A\pi(\ud x) \int
_{A^\complement} q(x,\ud y) \min\bigl\{1,r(x,y)\bigr\}
\\
&\le&(a+b)^2 \int_A \pi(\ud x) \bigl(1-
\rho(x) \bigr) \le(a+b)^2 p \Bigl(1-\inf_{x\in A}
\rho(x) \Bigr).
\end{eqnarray*}
Now, according to our choice of $a$ and $b$,
\[
(a+b)^2 p = \bigl(1-b^2(1-p) \bigr) +
2b^2(1-p) + b^2p = 1 + b^2 =
(1-p)^{-1}. 
\]

Consider then (ii).
The case $\Gap(P)=0$ is trivial, so assume $\Gap(P)>0$ and assume
the claim does not hold. Then there exists an $\varepsilon>0$ and
a set $A\in\mathcal{B}(\mathsf{X})$ with $p\defeq\P(A)\in(0,1)$
such that
$1-\rho(x)\le
\Gap(P)- \varepsilon$ for all $x\in A$. From (i),
$\Gap(P)\le(1-p)^{-1} (\Gap(P)-\varepsilon)$.
Because $\pi$ is not concentrated
on points, we may choose $p$ as small as we want, which leads to a
contradiction.
\end{pf}
%


\section{Lemmas for Sections \texorpdfstring{\protect\ref{141414141414}}{4} and \texorpdfstring{\protect\ref{151515151515}}{5}}
\label{141414141414-lemmas} 

%
\begin{lemma}
\label{lem:coupling-details} 
Suppose $X=(X_1,\ldots,X_n)$ and $Y=(Y_1,\ldots,Y_n)$ are Markov chains
on a common state space $(\mathsf{X},\B(\mathsf{X}))$
with kernels $P$ and $Q$, and initial distributions $\pi$ and
$\varpi$, respectively, which are invariant such that
$\pi P = \pi$ and $\varpi Q = \varpi$.
Then, the distributions of $X$ and $Y$ denoted
as $\mu_X$ and $\mu_Y$ satisfy the following inequality
for any $C\in\B(\mathsf{X})$:
\[
\|\mu_X -\mu_Y \| \le\| \pi- \varpi\| + 2(n-1)\pi
\bigl(C^\complement\bigr) + (n-1) \sup_{x\in C} \bigl\| P(x,\uarg)
- Q(x,\uarg) \bigr\|,
\]
where $\|\mu_X - \mu_Y\| \defeq\sup_{|f|\le1} |\mu_X(f)-\mu_Y(f)|$
denotes the total variation.
\end{lemma}
%

\begin{pf} 
Let $A\in\B(\mathsf{X})$. We shall use the shorthand notation
$x=x_{1:n} = (x_1,\ldots,x_n)$
and denote $g_P^{(1:n)}(x) \defeq\mathbbm{I} \{x\in A \}$,
\[
g_P^{(1:k)}(x_{1:k}) \defeq 
\int
P(x_k,\ud x_{k+1}) \cdots\int P(x_{n-1},\ud
x_n) \mathbbm{I} \{x\in A \},\qquad 2\le k\le n-1,
\]
and $g_P^{(1:1)} \defeq g_P^{(1)}$, and define $g_Q^{(\uarg)}$ similarly
using the kernel $Q$.

Note that $g_P^{(\uarg)}$ and $g_Q^{(\uarg)}$
take values between zero and one
and the total variation satisfies
$\| \pi- \varpi\| = 2 \sup_{0\le f\le1} |\pi(f)-\varpi(f)| = 2
\sup_{A\in\B(\mathsf{X})} |\pi(A)-\varpi(A)|$.
\begin{eqnarray*}
\bigl| \mu_X(A)-\mu_Y(A)\bigr | &= &\bigl|\pi\bigl(g_P^{(1)}
\bigr) - \varpi\bigl(g_Q^{(1)}\bigr)\bigr|
\\
&\le&\bigl|\pi\bigl(g_Q^{(1)}\bigr)-\varpi\bigl(g_Q^{(1)}
\bigr)\bigr| + \bigl|\pi\bigl(g_P^{(1)} - g_Q^{(1)}
\bigr)\bigr|
\\
&\le&\frac{1}{2}\| \pi- \varpi\| +\bigl |\pi\bigl(g_P^{(1)}
- g_Q^{(1)}\bigr)\bigr|,
\end{eqnarray*}
showing the claim for $n=1$. Assume then $n\ge2$ and observe that we
can write
$|\pi(g_P^{(1)} - g_Q^{(1)})| = |\E[g_P^{(1)}(X_1) - g_Q^{(1)}(X_1)]|$.
We may continue inductively
\begin{eqnarray*}
&& \bigl|\E\bigl[ \bigl(g_P^{(1:n-1)} - g_Q^{(1:n-1)}
\bigr) (X_{1:n-1})\bigr] \bigr|
\\
&&\qquad\le \bigl| \E\bigl[ \bigl(g_P^{(1:n)} - g_Q^{(1:n)}
\bigr) (X_{1:n})\bigr] \bigr| + \biggl|\E \biggl[ \int \Delta(X_{n-1},\ud
x_n) g_Q^{(1:n)}(X_{1:n-1},x_n)
\biggr] \biggr|,
\end{eqnarray*}
where $\Delta(x,\ud y) \defeq P(x,\ud y) - Q(x,\ud y)$, and observe that
\begin{eqnarray*}
&&\biggl|\E \biggl[ \int\Delta(X_{n-1},\ud x_n)
g_Q^{(1:n)}(X_{1:n-1},x_n) \biggr] \biggr|
\\
&&\qquad \le\P(X_{n-1}\notin C) + \sup_{x_{1:n-2}\in\mathcal{X}^{n-2}} \sup
_{x_{n-1}\in C} \biggl| \int\Delta(x_{n-1},\ud x_n)
g_Q^{(1:n)}(x_{1:n})\biggr |
\\
&&\qquad\le\pi\bigl(C^\complement\bigr) + \frac{1}{2} \sup
_{x\in C} \bigl\| P(x,\uarg)-Q(x,\uarg) \bigr\|,
\end{eqnarray*}
because $|\int\Delta(X_{n-1},\ud x_n)
g_Q^{(1:n)}(X_{1:n-1},x_n)|\le1$ and $0\le g_Q^{(1:n)}\le1$.
\end{pf}
%

%
\begin{lemma}
\label{lem:imh-subgeom-generic} 
Assume $q\gg\pi$ and denote $\mu(x)\defeq\pi(\ud x)/ q(\ud x)$.
Suppose that there exists a strictly increasing
$\phi\dvtx(0,\infty)\to[1,\infty)$ with
$\liminf_{t\to\infty} \phi(t)/\break t>0$, such that
%
\begin{equation}
\int\pi(\ud x) \phi \bigl(\mu(x) \bigr) <\infty. \label{eq:imh-bound-generic}
\end{equation}
Then, there exist constants $M,c,\varepsilon\in(0,\infty)$
and a probability measure $\nu$ on
$ (\mathsf{X},\mathcal{B}(\mathsf{X}) )$
such that for the independent Metropolis--Hastings $P$,
%
\begin{eqnarray}
P V(x) &\le& V(x) -c V(x)/\phi^{-1} \bigl(V(x) \bigr)\qquad \mbox{if }\mu(x)>
M,\label{eq:imh-drift-generic}
\\
P(x; \uarg) &\ge&\varepsilon\nu(\uarg)\qquad \mbox{if } \mu(x) \le M, \label{eq:imh-mino-generic}
\end{eqnarray}
and $\nu(V)<\infty$,
where $V(x)\defeq\phi (\mu(x) )$.
\end{lemma}
%

\begin{pf} 
Denote $A_{x} \defeq \{y \in\mathsf{X}\given
\frac{\mu(y)}{\mu(x)}\ge1 \}$ and $R_{x}\defeq
A_{x}^\complement$ and write
\begin{eqnarray*}
P V(x) & =& \int_{A_{x}}\frac{V(y)}{\mu(y)} \pi(\ud y) +\int
_{R_{x}}\frac{V(y)}{\mu(x)}\pi(\ud y) +V(x,w)\int
_{R_{x}} \biggl(1- \frac{\mu(y)}{\mu(x)} \biggr) q(\ud y)
\\
&\le&\frac{1}{\mu(x)} \int\pi(\ud y) V(y) +V(x) \biggl(1-\frac{\pi(R_{x})}{\phi^{-1} (V(x) )}
\biggr),
\end{eqnarray*}
because $\mu(y)\ge\mu(x)$ on $A_{x,w}$.
The first term on the right
vanishes and $\pi(R_{x})\to1$
as $\mu(x)\to\infty$, and $\liminf_{u\to\infty} u/\phi^{-1}(u) > 0$,
implying \eqref{eq:imh-drift-generic}. For \eqref{eq:imh-mino-generic},
observe that for $\mu(x)\le M$,
\[
P(x, B) \ge\int_B \min \biggl\{\frac{1}{M},
\frac{1}{\mu(y)} \biggr\} \pi(\ud y) \eqdef\tilde{\nu}(B),
\]
and we can take $\varepsilon= \tilde{\nu}(\mathsf{X})$
and $\nu= \varepsilon^{-1} \tilde{\nu}$,
for which \eqref{eq:imh-bound-generic} implies $\nu(V)<\infty$.
\end{pf}
%


\section{Lemmas for Section \texorpdfstring{\protect\ref{171717171717}}{7}}
\label{171717171717-lemmas} 

We denote by $n(x) \defeq x/|x|$ the unit vector pointing in the
direction of $x\neq0$ and by $B(x,r)\defeq\{y\in\R^d\dvtx|x-y|\le
r\}$
the (closed) Euclidean ball.

%
\begin{lemma}
\label{lem:annulus-containment} 
Assume $\pi$ satisfies Condition \ref{a:super-exp-regular}, and that
$c\dvtx\mathsf{X}\to[1,\infty)$ satisfies $\limsup_{|x|\to\infty}
c(x)e^{-|x|}<\infty$. Then, there exist constants $M,b\in[1,\infty)$ such
that for all $|x|\ge M$,
\[
D_x \defeq \biggl\{y\in\R^d\dvtx\frac{1}{c(x)} \le
\frac{\pi(y)}{\pi(x)} \le c(x) \biggr\} \subset B\bigl(0,b|x|\bigr)\setminus B
\bigl(0,b^{-1}|x|\bigr).
\]
\end{lemma}
%

\begin{pf} 
Let $c'> \limsup_{|x|\to\infty} c(x)e^{-|x|}$.
Choose any $C\in(4c',\infty)$ and let
$M_0\in[1\vee\log c',\infty)$ be sufficiently large so that
there exists
a $\beta_\pi\in(0,1]$ such that for all $|x|\ge M_0$,
\[
c(x) \le c' e^{|x|},\qquad n(x) \cdot\nabla\log \pi(x) \le-C
\quad\mbox{and}\quad n(x) \cdot n\bigl(\nabla\pi(x)\bigr) < -\beta_\pi.
\]
Let $\delta\in(0,1)$. Then for any $|x|\ge M_0(1 - \delta)^{-1}$ and all
$z = t n(x)$ with
$|t|\le\delta$, we have
%
\begin{equation}
\biggl\llvert \log\frac{\pi(x+z)}{\pi(x)}\biggr\rrvert = |t|\int_0^{1}
\bigl|n(x+\lambda z) \cdot\nabla\log\pi (x + \lambda z)\bigr| \,\ud\lambda
\ge C |t|. \label{eq:gradient-stuff}
\end{equation}
Now, if $|x|> aM_0$ where $a \defeq\exp(2\pi\tan(\arccos(\beta
_\pi)))$,
then \cite{saksman-vihola}, Lemma~22, implies
%
\begin{equation}
\bigl\{y\in\R^d\dvtx\pi(y)=\pi(x)\bigr\} \subset B\bigl(0,a|x|\bigr)\setminus
B\bigl(0,a^{-1}|x|\bigr). \label{eq:contour-containment}
\end{equation}
Take any $M > 4a M_0$, and choose $|x|\ge M$. Then,  condition
\eqref{eq:gradient-stuff} implies that any $z = \lambda x\in D_x$,
where $\lambda>0$ satisfies
\[
\bigl|(\lambda- 1)|x| \bigr|\le C^{-1}\log c(x) \le C^{-1} \bigl(\log
\bigl(c'\bigr) + |x|\bigr) \le2 C^{-1} |x|.
\]
We deduce that $|\lambda-1|< 1/2$. Again, using
\eqref{eq:contour-containment}, we deduce that the claim holds with
$b=2a$.
\end{pf}
%

%
\begin{lemma}
\label{lem:containment-and-delta-acc} 
Assume $\pi$ satisfies Condition \ref{a:super-exp-regular}.
\begin{longlist}[(iii)]
\item[(i)]
Then, for any constant $\nu\in(0,\infty)$, there exists
a constant $b_\nu\in[1,\infty)$ such that for all $x\in\mathsf{X}$,
$ \{x+z\given\frac{\pi(x+z)}{\pi(x)}\ge\nu \}
\subset B (0,b_\nu(|x|\vee1) )$.
Assume also that $q$ satisfies Condition \ref{a:super-exp-regular}.
\item[(ii)]
There exists a constant $\nu\in(0,\infty)$ such that $\inf_{x\in
\mathsf{X}}
q ( \{z\given\frac{\pi(x+z)}{\pi(x)}\ge\nu \} )>0$.
\item[(iii)]
For any constant $\nu\in(0,\infty)$,
there exists a constant $M=M(\nu)\in[1,\infty)$ such that
$\inf_{|x|\ge M} q ( \{z\given\frac{\pi(x+z)}{\pi(x)}\ge
\nu
\} )>0$.
\end{longlist}
\end{lemma}
%

\begin{pf} 
Consider first (i).
The existence of such a finite
constant follows for $x$ in compact sets by the continuity of $\pi$
and in the tails by Lemma~\ref{lem:annulus-containment}.

The claim (ii) follows on compact sets by the
continuity of $\log\pi$, and in the tails
as in \cite{jarner-hansen}, proof of Theorem~4.3; the last claim
(iii) follows similarly.
\end{pf}
%

When the target and the proposal distributions satisfy also
Condition \ref{cond:strong-super-exp-regular}, we have
a decay rate for $q(D_x)$.

\begin{lemma}
\label{lem:q-d-x} 
Assume Condition \ref{cond:strong-super-exp-regular}, and assume
$\limsup_{|x|\to\infty} c(x) e^{-|x|} < \infty$. Then, for any
$\varepsilon'>0$ there exists a constant $M_0\in[M,\infty)$ such
that for
all $|x|\ge M_0$,
\[
q (D_x ) \le\varepsilon' \frac{\log(c(x))}{|x|^{\rho-1}}\qquad
\mbox{where } D_x\defeq \biggl\{z\in\R^d\given
\frac{1}{c(x)} \le\frac{\pi
(x+z)}{\pi
(x)} \le c(x) \biggr\}.
\]
\end{lemma}
%

\begin{pf} 
Lemma~\ref{lem:annulus-containment}
implies
$b\in[1,\infty)$ and $M'\in[1,\infty)$ such that for all $|x|\ge M'$
the annulus $D_x \subset B(0,b|x|)\setminus B(0,b^{-1}|x|)$.
This implies that for any constant $c_\ell\in[1,\infty)$ one can choose
$M_\ell\in[M',\infty)$ such that
\[
n(x+z) \cdot\nabla \log\pi(x+z) \le-c_\ell|x+z|^{\rho-1}
\qquad\mbox{for all } |x|\ge M_\ell, z\in D_x.
\]
Denoting $\ell(x)\defeq\log\pi(x)$, we write
\[
D_x = \bigl\{z\in\R^d\given\bigl|\ell(x+z)-\ell(x)\bigr|\le\log
c(x)\bigr\}.
\]
Define the contour surface set
$S_{\pi(x)} \defeq\{y\in\R^d\given\pi(y)=\pi(x)\}$ and
\[
C_{\pi(x)}(\delta)\defeq \bigl\{y+t n(y)\given y\in S_{\pi(x)},
|t|\le\delta \bigr\}.
\]
We will now check that with our conditions, for
$|x|\ge M_\ell b$,
%
\begin{equation}
D_x+x \subset C_{\pi(x)}(\delta_x) \qquad\mbox{where }
\delta_x \defeq \frac{b^{\rho-1}}{c_\ell} \cdot\frac{\log
c(x)}{|x|^{\rho-1}}.
\label{eq:annulus-containment}
\end{equation}
Because $D_x+x=D_y+y$ whenever $\pi(x)=\pi(y)$,
it is sufficient to consider $z\in D_x$ such that $z = t n(x)$
As in the proof of Lemma~\ref{lem:annulus-containment},
\begin{eqnarray*}
\bigl|\ell (x+z )-\ell(x)\bigr| 
&=& |t|\int
_0^1 \bigl\llvert n(x + \lambda z) \cdot \nabla
\ell (x+\lambda z ) \bigr\rrvert \,\ud\lambda
\\
&\ge&|t| c_\ell|x|^{\rho-1}\int_0^1
\biggl\llvert 1 + \frac{t}{|x|}\biggr\rrvert ^{\rho-1} \,\ud t \ge
c_\ell b^{-(\rho-1)} |x|^{\rho-1} |t|.
\end{eqnarray*}
Now $|\ell (x+z )-\ell(x)|\le\log c(x)$ implies
\eqref{eq:annulus-containment}.

Write then, by Fubini's theorem,
\begin{eqnarray*}
q(D_x) 
&\le&\int_{C_{\pi(x)}(\delta_x)-x} \bar{q}(z) \,\ud
z
\\
&=& \int_0^{\bar{q}(0)} \mathcal{L}^d
\bigl(z\in\R^d\given\bar{q}\bigl(|z|\bigr)\ge t, z\in C_{\pi(x)}(
\delta_x)-x \bigr) \,\ud t
\\
&=& \int_0^\infty \mathcal{L}^d
\bigl(z\in\R^d\given|z|\le u, z\in C_{\pi(x)}(
\delta_x)-x \bigr) \bigl|\bar{q}'(u)\bigr|\,\ud u.
\end{eqnarray*}
Now, \cite{jarner-hansen}, proof of Theorem~4.1, shows that
for $u\le|x|/2$,
\begin{eqnarray*}
\mathcal{L}^d \bigl(C_{\pi(x)}(\delta_x)\cap
B(x,u) \bigr) \le\delta_x \biggl(\frac{|x|+u}{|x|-u}
\biggr)^{d-1} \frac{\mathcal{L}^d (B(x,3u) )}{u} \le3^{2d-1}c_d
\delta_x u^{d-1}, 
\end{eqnarray*}
where $c_d = \mathcal{L}^d(B(0,1))$.
By polar integration,
\begin{eqnarray*}
\mathcal{L}^d \bigl(C_{\pi(x)}(\delta_x) \bigr)
&\le& c_d \sup_{y\in S_{\pi(x)}} \int_{|y|-\delta_x}^{|y|+\delta_x}
r^{d-1} \,\ud r \\
&\le& 2 c_d b^{d-1}
\delta_x |x|^{d-1} \le4 c_d b^{d-1}
\delta_x u^{d-1},
\end{eqnarray*}
where the latter inequality holds for $u\ge|x|/2$.
We obtain
\[
q(D_x) \le c'\delta_x \biggl(1 + \int
_0^\infty u^d \bigl|\hat{q}'(u)\bigr|
\,\ud u \biggr),
\]
and because $\hat{q}$ is monotone decreasing,
integration by substitution yields
\[
\int_0^M u^d \bigl|
\hat{q}'(u)\bigr| \,\ud u = d \int_0^M
u^{d-1} \hat{q}(u) \,\ud u - M^d \hat{q}(M) \le d
c_d^{-1} \int\hat{q}(x) \,\ud x <\infty.
\]
We deduce
$q(D_x)\le c'' \delta_x$ and conclude by choosing $c_\ell$
sufficiently large.
\end{pf}
%
\end{appendix}

\section*{Acknowledgement} 
We thank Anthony Lee for fruitful discussions.



%



\printaddresses
\end{document}